\documentclass[11pt]{amsart}

\usepackage{enumerate}
\usepackage{amsmath}%
\usepackage{amsfonts}%
\usepackage{amssymb}%
\usepackage{graphicx}
\usepackage{pgfplots}
\usepackage{mathrsfs}
\usepackage{hyperref}
\usepackage{fullpage}

\newtheorem{theorem}{Theorem}
\theoremstyle{plain}

\newtheorem{claim}[theorem]{Claim}

\newtheorem{conjecture}[theorem]{Conjecture}

\newtheorem{corollary}[theorem]{Corollary}

\newtheorem{fact}[theorem]{Fact}
\newtheorem{lemma}[theorem]{Lemma}

\newtheorem{proposition}[theorem]{Proposition}

\numberwithin{equation}{section}
\numberwithin{theorem}{section}
\numberwithin{case}{section}

\numberwithin{subcase}{case}

\def\A{\mathcal{A}}

\def\F{\mathcal{F}}

\def\K{R}

\def\cP{\mathcal{P}}
\def\Q{\mathcal{Q}}
\def\R{\mathcal{R}}
\def\sss{\mathcal{S}}

\def \a{\alpha}
\def \e{\epsilon}
\def \r{\gamma}
\def\bfv{\mathbf{v}}
\def\bfi{\mathbf{i}}
\def\bfu{\mathbf{u}}

\begin{document}

\title{Minimum codegree threshold for $C_6^3$-factors in $3$-uniform Hypergraphs }

\author{Wei Gao}
\address{Dept of Math and Stat, Georgia State University, Atlanta, GA 30302-4110, USA}
\email[Wei Gao]{wgao2@gsu.edu}

\author{Jie Han}
\address{Instituto de Matem\'{a}tica e Estat\'{\i}stica, Universidade de S\~{a}o Paulo, Rua do Mat\~{a}o 1010, 05508-090, S\~{a}o Paulo, Brazil}
\email[Jie Han]{jhan@ime.usp.br}
\thanks{The second author is supported by FAPESP (Proc. 2014/18641-5).}

\begin{abstract}
Let $C_6^3$ be the 3-uniform hypergraph on $\{1,\dots, 6\}$ with edges $123, 345,561$, which can be seen as the triangle in 3-uniform hypergraphs.
For sufficiently large $n$ divisible by 6, we show that every $n$-vertex 3-uniform hypergraph $H$ with minimum codegree at least $n/3$ contains a $C_6^3$-factor, i.e., a spanning subhypergraph consisting of vertex-disjoint copies of $C_6^3$.
The minimum codegree condition is best possible. This improves the asymptotical result obtained by Mycroft and answers a question of R\"odl and Ruci\'nski exactly.
\end{abstract}

\maketitle

\section{Introduction}

In graph theory, finding certain large or spanning subgraph in a given graph $H$ is one of the most important topics to study. 
In particular, finding vertex-disjoint copies of some given graph has a long history and has received much attention (see surveys \cite{KuOs-survey, RR, zsurvey}).
More precisely, given a graph $G$ of order $g$ and a graph $H$ of order $n$, a \emph{$G$-tiling} of $H$ is a subgraph of $H$ that consists of vertex-disjoint copies of $G$. When $g$ divides $n$, a \emph{perfect $G$-tiling} (or a \emph{$G$-factor}) of $H$ is a $G$-tiling of $H$ consisting of $n/g$ copies of $G$. 

When $G$ is a single edge, the perfect $G$-tiling is also called a \emph{perfect matching}.
Tutte's Theorem \cite{Tu47} gives a characterization of all those graphs which contain a perfect matching.
But for the tilings of general $G$, no such characterization is known.
Moreover, Hell and Kirkpatrick \cite{HeKi} showed that the decision problem of whether a graph $H$ has a $G$-factor is NP-complete if and only if $G$ has a component which contains at least 3 vertices.
So it is natural to find sufficient conditions which ensure the existence of a $G$-factor.

The celebrated Hajnal-Szemer\'edi Theorem \cite{HaSz} says that every $n$-vertex graph $H$ with $\delta(H)\ge (k - 1)n/k$ contains a $K_k$-factor (the case $k=3$ was obtained by Corr\'adi and Hajnal \cite{CoHa}).
For general graph $G$, the minimum degree threshold for $G$-factors was determined by K\"uhn and Osthus \cite{KuOs09}, up to an additive constant, improving the results in \cite{AY96, KSS-AY}.

It is natural to extend these results to hypergraphs. 
Given $k\ge 2$, a \emph{$k$-uniform hypergraph} (in short, \emph{$k$-graph}) consists of a vertex set $V$ and an edge set $E\subseteq \binom{V}{k}$, where every edge is a $k$-element subset of $V$. Given a $k$-graph $H$ with a set $S$ of $d$ vertices (where $1\le d\le k-1$) we define $\deg_{H}(S)$ to be the number of edges containing $S$ (the subscript $H$ is often omitted if it is clear from the context). The \emph{minimum $d$-degree} $\delta_d(H)$ of $H$ is the minimum of $\deg_{H}(S)$ over all $d$-vertex sets $S$ in $H$.
We refer to $\delta _{k-1} (H)$ as the \emph{minimum codegree} of $H$.
The $G$-tilings and $G$-factors in $k$-graphs are defined analogously as in graphs. Define $t_d(n, G)$ to be the smallest integer $t$ such that every $k$-graph $H$ of order $n\in g\mathbb{N}$ with $\delta_d(H)\ge t$ contains a $G$-factor.

However, tilings problems become much harder for hypergraphs.
For example, despite much recent progress \cite{AFHRRS, CzKa, Khan2, Khan1, KOT, RRS09, TrZh13, TrZh15}, we still do not know the 1-degree threshold for a perfect matching in $k$-graphs for arbitrary $k$.

Other than the matching case, most work on $G$-factors has been done when $G$ is a 3-graph on four vertices or when $G$ is a $k$-partite $k$-graph. Let $K_4^3$ be the complete 3-graph on four vertices, and let $K_4^3-e$ be the (unique) 3-graph on four vertices with three edges. 
Let $K_4^3 - 2e$ be the (unique) 3-graph on four vertices with two edges (this 3-graph was denoted by $C_4^3$ in \cite{KO} and by $\mathcal{Y}$ in \cite{HZ1}). 
Lo and Markstr\"om \cite{LM1} showed that $t_2(n, K_4^3)=(3/4+o(1)) n$ and Keevash and Mycroft \cite{KM1} determined the exact value of $t_2(n, K_4^3)$ for sufficiently large $n$. 
Lo and Markstr\"om \cite{LM2} showed that $t_2(n, K_4^3 - e)=(1/2+o(1))n$, and very recently Han, Lo, Treglown and Zhao \cite{HLTZ_K4} showed that $t_2(n, K_4^3 - e)=n/2-1$ for large $n$. 
K\"uhn and Osthus \cite{KO} showed that $t_2(n, K_4^3 - 2e)=(1+o(1))n/4$, and Czygrinow, DeBiasio and Nagle \cite{CDN} subsequently determined $t_2(n, K_4^3 - 2e)$ exactly for large $n$. More recently Han and Zhao \cite{HZ3} and independently Czygrinow \cite{Czy14} determined $t_1(n, K_4^3 - 2e)$ exactly for large $n$. 
Mycroft \cite{My14} determined $t_{k-1}(n, F)$ asymptotically for many $k$-partite $k$-graphs $F$ (including complete $k$-partite $k$-graphs and loose cycles). Han, Zang and Zhao \cite{HZZ_tiling} determined $t_1(n, K)$ asymptotically for all complete $3$-partite $3$-graphs $K$.

A \emph{$k$-uniform loose cycle $C_s^k$} is an $s$-vertex $k$-graph whose vertices can be ordered cyclically in such a way that the edges are sets of consecutive $k$ vertices and every two consecutive edges share exactly one vertex. Note that by definition, $s$ must be divisible by $k-1$ and at least $3k-3$.

R\"odl and Ruci\'nski \cite[Problem 3.15]{RR} asked for the values of $t_2(n, C_s^3)$ for all $s\ge 6$.
This was solved asymptotically by the aforementioned result of Mycroft \cite{My14}, who determined $t_{k-1}(n, C_s^k)$ asymptotically for $k\ge 3$.
In particular, it \cite{My14} is shown that $t_2(n, C_6^3) = (1/3 + o(1))n$.
In this paper we determine the exact value of $t_2(n, C_6^3)$ for sufficiently large $n$, improving Mycroft's result.

\begin{theorem}[Main result]\label{main}
Let $n\in 6\mathbb{Z}$ be sufficiently large. Suppose $H$ is a $3$-graph on $n$ vertices with 
$\delta_2(H) \ge n/3$. Then $H$ contains a $C_6^3$-factor.
\end{theorem}

The minimum codegree condition in Theorem~\ref{main} is best possible by the following example.
Let $n$ be an integer divisible by 6 and $V=X\cup Y$ such that $|V|=n$, $X\cap Y=\emptyset$, $|X|=n/3-1$.
Let $H_0$ be the 3-graph on $V$ where $E(H_0)$ consists of all triples that intersect $X$.
Clearly, $\delta_2(H_0) = |X| = n/3-1$.
Moreover, observe that each copy of $C_6^3$ in $H_0$ must contain at least two vertices in $X$.
Thus, $H_0$ does not contain a $C_6^3$-factor.

The approach we use in this paper is quite different from that used by Mycroft \cite{My14}.
Indeed, the main tool in \cite{My14} is the \emph{Hypergraph Blow-up Lemma} recently developed by Keevash \cite{Keevash_blowup}. 
In contrast, our proof uses the so-called \emph{lattice-based absorbing method}, together with an \emph{almost perfect tiling lemma} and an {extremal case analysis}.
The lattice-based absorbing method is developed recently by the second author, which is a variant of the absorbing method initiated by R\"odl, Ruci\'nski and Szemer\'edi \cite{RRS06}.
Roughly speaking, given a $k$-graph $H$, the existence of the so-called \emph{absorbing set} relies on the fact that $V(H)$ is closed (see definitions in Section 3). 
However, $V(H)$ is not closed in some applications.
In this case, the lattice-based absorbing method provides a weaker absorbing set, which sometimes is sufficient by combining other information of $H$.
Interestingly, in our problem, it is not clear whether $V(H)$ is closed (it will certainly be true if $\delta_2(H)\ge (1/2+o(1))n$).
Nevertheless, the lattice-based absorbing method works well and gives the absorbing set.

In the forthcoming paper \cite{HZ_cycpac}, the second author and Zhao determine $t_{k-1}(n, C_s^k)$ exactly, improving the asymptotical result of Mycroft.
The reason for splitting the results into two papers is because the techniques used are different.
In fact, since $C_6^3$ has a unique 3-partite realization, which is balanced ($C_6^3$ is a spanning subhypergraph of $K_{3}^3(2)$), the proof of our almost perfect tiling lemma is a standard application of the regularity method.
In contrast, any other $C_s^k$ allows an unbalanced $k$-partite realization.
This makes the proof of the almost perfect tiling lemma more involved.
In contrast, the proof of the absorbing lemma in \cite{HZ_cycpac} becomes easier.

\section{Proof of Theorem~\ref{main}}

As a typical approach to obtain exact results, our proof of Theorem \ref{main} consists of an extremal case and a nonextremal case.
For $k\ge 3$ and $\e>0$, we say that a $k$-graph $H$ is \emph{$\e$-extremal} if there is a vertex set $S\subseteq V(H)$ of size $\lfloor \frac{k-1}kn \rfloor$ such that $e(H[S]) \le \e n^k$.

\begin{theorem}[Nonextremal case]\label{nonex}
Let $\gamma > 0$ and let $n\in 6\mathbb{Z}$ be sufficiently large. Suppose $H$ is an $n$-vertex $3$-graph with $\delta_2(H) \ge n/3 - \r n$. If $H$ is not $3\r$-extremal, then $H$ contains a $C_6^3$-factor.
\end{theorem}

\begin{theorem}[Extremal case]\label{thm:E}\
Let $0<\e\ll 1$ and let $n\in 6\mathbb{Z}$ be sufficiently large. Suppose $H$ is an $n$-vertex $3$-graph with 
$\delta_2(H) \ge n/3$. If $H$ is $\e$-extremal, then $H$ contains a $C_6^3$-factor.
\end{theorem}

Theorem \ref{main} follows from Theorem \ref{nonex} and \ref{thm:E} immediately by choosing $\e$ from Theorem \ref{thm:E}.

As mentioned in Section 1, in the proof of Theorem~\ref{nonex}, we use the {lattice-based absorbing method}. 
Here is our absorbing lemma.

\begin{lemma}[Absorbing]\label{lem:absorb}
Given $0<\r \ll 1$, there exists $\alpha>0$ such that the following holds for sufficiently large $n$. Suppose $H$ is an $n$-vertex $3$-graph such that $\delta_2(H)\ge (1/3-\r)n$. 
Then there exists a vertex set $W\subseteq V(H)$ with $|W|\leq \r n$ such that for any vertex set $U\subseteq V(H)\setminus W$ with $|U|\leq \alpha n$ and $|U|\in 6\mathbb{Z}$, both $H[W]$ and $H[U\cup W]$ contain $C_6^3$-factors.
\end{lemma}

By Lemma~\ref{lem:absorb}, the task is reduced to finding an almost $C_6^3$-factor in the $3$-graph $H$ after removing the absorbing set $W$. In fact, we prove a more general tiling result in the following lemma.
For integers $k,h>0$, let $K_k^k{(h)}$ be a complete $k$-partite $k$-graph with $h$ vertices in each part.

\begin{lemma}[Almost perfect tiling]\label{lem:tiling}
Let $\gamma,\a > 0$, $h\in \mathbb{Z}$ and let $n$ be a sufficiently large integer. Suppose $H$ is an $n$-vertex $k$-graph with $\delta_{k-1}(H) \ge \frac{n}{k}-\gamma n$. If $H$ is not $2 \gamma$-extremal, then $H$ contains a $K_k^k{(h)}$-tiling that leaves at most $2\a n$ vertices uncovered.
\end{lemma}

\begin{proof}[Proof of Theorem \ref{nonex}.]
Apply Lemma~\ref{lem:absorb} and get a vertex set $W$ of order at most $\r n$ with the absorbing property.
Let $V' = V(H)\setminus W$ and $H' = H[V']$.
Note that $\delta_2(H') \ge n/3 - \r n - |W| \ge (1/3 - 2\r) |V'|$.
We claim that $H'$ is not $2\r$-extremal.
Indeed, suppose $H'$ is $2\r$-extremal, i.e., there exists a vertex set $S\subseteq V'$ of size $2|V'|/3$ such that $e_{H'}(S)\le 2\r |V'|^3$. Then by adding arbitrary $2(n - |V'|)/3 \le \r n$ vertices of $H\setminus S$ to $S$, we get a set $S'\subseteq V(H)$ of order $2n/3$ with $e_H(S') \le 2\r |V'|^3 + \r n \cdot n^{2}\le 3\r n^3$.
This means that $H$ is $3\r$-extremal, a contradiction.
So we can apply Lemma~\ref{lem:tiling} on $H'$ with $\r$, $\a/2$, $k=3$ and $h=2$.
This gives a $K_3^3(2)$-tiling of $H'$, which can be treated as a $C_6^3$-tiling of $H'$, that leaves at most $\a n$ vertices uncovered.
These uncovered vertices can be absorbed by $W$ and thus we get a $C_6^3$-factor of $H$.
\end{proof}

The rest of the paper is organized as follows. We give an outline of the proof of Lemma~\ref{lem:absorb} and then prove Lemma \ref{lem:absorb} in Section 3. We prove Lemma \ref{lem:tiling} in Section 4 and the extremal case, Theorem \ref{thm:E} in Section 5, respectively.
We also give some concluding remarks at the end of the note.

\medskip
\noindent\textbf{Notations.}
Throughout the rest of the paper, we write $C_6$ instead of $C_6^3$ for short.
For a $k$-graph $H$ and $A\subseteq V(H)$, we write $e_H(A)$ for $e(H[A])$, the number of edges in $H$ induced on $A$.
Moreover, for not necessarily distinct sets $A_1,\dots, A_k$, we write $e_H(A_1,\dots, A_k)$ as the number of edges $\{v_1,\dots, v_k\}$ in $H$ such that $v_i\in A_i$ for all $i\in [k]$.
Furthermore, for vertices $u,v$ and a vertex set $S$, let $\overline \deg_H (v, S)=\binom{|S\setminus \{v\}|}2-\deg_H (v,S)$ and $\overline \deg_H (uv, S)=|S\setminus \{u, v\}| -\deg_H (uv, S)$.
The subscript is often omitted if it is clear from the context.
Throughout this paper, $x\ll y$ means that for any $y\ge 0$ there exists $x_0\ge 0$ such that for any $x\le x_0$ the following statement holds.
Similar notations with more constants are defined similarly.

\section{Proof of the Absorbing Lemma}

\subsection{Preliminary and an outline of the proof}
Following the previous work by the absorbing method, we use the so-called reachability argument.
More precisely, for vertices $x, y$ in an $n$-vertex 3-graph $H=(V, E)$ and a set $S\subseteq V\setminus \{x,y\}$, we call $S$ a \emph{reachable $|S|$-set} for $x$ and $y$ if both $H[\{x\}\cup S]$ and $H[\{y\} \cup S]$ contain $C_6$-factors.
We say two vertices $x, y$ are \emph{$(\beta, i)$-reachable in $H$} if there are at least $\beta n^{6i-1}$ reachable $(6i-1)$-sets for $x$ and $y$ in $H$.
We say a vertex set $U$ is \emph{$(\beta, i)$-closed in $H$} if every two vertices of $U$ are $(\beta, i)$-reachable.
For $x\in V$, let $\tilde{N}_{\beta, i}(x)$ be the set of vertices that are $(\beta, i)$-reachable to $x$.

We use some notations in \cite{KM1}.
For an integer $r\ge 1$, let $\cP=\{V_1,\dots, V_r\}$ be a partition of $V$. 
The \emph{index vector} $\mathbf{i}_{\cP}(S)\in \mathbb{Z}^r$ of a subset $S\subset V$ with respect to $\cP$ is the vector whose coordinates are the sizes of the intersections of $S$ with each part of $\cP$, i.e., $\mathbf{i}_{\cP}(S)_{V_i}=|S\cap V_i|$ for $i\in [r]$. 
We call a vector $\mathbf{i}\in \mathbb{Z}^r$ an \emph{$s$-vector} if all its coordinates are nonnegative and their sum equals $s$. 
Given $\mu>0$, a 3-vector $\mathbf{v}$ is called a $\mu$\emph{-robust edge-vector} if at least $\mu n^3$ edges $e\in E$ satisfy $\mathbf{i}_\cP(e)=\mathbf{v}$. 
A $6$-vector $\mathbf{v}$ is called a \emph{$\mu$-robust $C_6$-vector} if there are at least $\mu n^6$ copies $K$ of $C_6$ in $H$ satisfy $\mathbf{i}_\cP(V(K))=\mathbf{v}$. 
Let $I_{\cP}^{\mu}(H)$ be the set of all $\mu$-robust edge-vectors and let $I_{\cP, C}^{\mu}(H)$ be the set of all $\mu$-robust $C_6$-vectors.
For $j\in [r]$, let $\mathbf{u}_j\in \mathbb{Z}^r$ be the $j$-th \emph{unit vector}, namely, $\bfu_j$ has 1 on the $j$-th coordinate and 0 on other coordinates.
A \emph{transferral} is the a vector of form $\bfu_i - \bfu_j$ for some distinct $i,j\in [r]$.
Let $L_{\cP, C}^{\mu}(H)$ be the lattice (i.e., the additive subgroup) generated by $I_{\cP, C}^{\mu}(H)$ (though $L_{\cP, C}^{\mu}(H)$ will not be explicitly used in the proof).

The proof of Lemma~\ref{lem:absorb} proceeds as follows.
Given an $n$-vertex 3-graph $H=(V, E)$ with $\delta_2(H)\ge (1/3 - \r)n$.
We first show that (Lemma~\ref{lem:partition}) there exists some $\beta$, $i$ (independent of $n$) such that $V$ admits a partition $\cP$ of at most three parts, such that each part is $(\beta, i)$-closed.
Next we show that (Lemma~\ref{lem:trans}) if $L_{\cP, C}^{\mu}(H)$ contains a transferral $\bfu_i - \bfu_j$, then $V_i\cup V_j$ is closed.
In this case we combine these two parts and thus reduce the number of parts in the partition.
We repeatedly merge parts until there is no transferral in $\cP$ (let us still call the new partition obtained from merging $\cP$).
Then we show that (Lemma~\ref{lem:even}) if $\cP$ contains no transferral, then all 6-vectors with all coordinates even must be present in $I_{\cP, C}^{\mu}(H)$.
Although by our assumption, there is no robust $C_6$-vector with odd coordinates in $\cP$ (this together with some vector above will give a transferral), we can show that (Lemma~\ref{lem:oddcopy}) there exists one copy of $C_6$ with odd coordinates, which turns out to be sufficient for the absorption (see the proof of Lemma~\ref{lem:absorb}).

As mentioned in Section 1, in most of existing applications of the absorbing method, it is shown that $V(H)$ is $(\beta', i')$-closed for some $\beta'>0$ and integer $i'$, which implies the absorbing lemma easily. It is interesting to know whether this holds for our problem.

\subsection{Lemmas}

We use the following result from \cite{LM1}.

\begin{proposition}[\cite{LM1}, Proposition 2.1] \label{prop:21}
For $\beta, \e>0$ and integers $i_0'>i_0$, there exists $\beta'>0$ such that the following holds for sufficiently large $n$. Given an $n$-vertex $3$-graph $H$ and a vertex $x\in V(H)$ with $|\tilde{N}_{\beta, i_0}(x)|\ge \e n$, then $\tilde{N}_{\beta, i_0}(x)\subseteq \tilde{N}_{\beta', i_0'}(x)$.
In other words, if $x,y\in V(H)$ are $(\beta,i_0)$-reachable in $H$ and $|\tilde{N}_{\beta, i_0}(x)|\ge \e n$, then $x,y$ are $(\beta',i_0')$-reachable in $H$.
\end{proposition}

We show that for every vertex $x$, $|\tilde{N}_{\beta, 1}(x)|$ is large in the next proposition.

\begin{proposition}\label{prop:Nv}
Fix $0< \beta \ll b$ and let $n$ be sufficiently large.
Suppose $H=(V, E)$ is a 3-graph on $n$ vertices satisfying $\delta_{2}(H)\ge b n$. For any $x\in V(H)$, $|\tilde{N}_{\beta, 1}(x)| \ge (b - \sqrt[3]\beta) n$.
\end{proposition}

\begin{proof}
Fix a vertex $x\in V$, we claim that for any vertex $x'\neq x$, $x'\in \tilde{N}_{\beta, 1}(x)$ if $|N(x)\cap N(x')|\ge \sqrt{\beta} n^{2}$.
Indeed, let $\{y, z\}$ be a pair in $N(x)\cap N(x')$. 
Pick a vertex $u\in V\setminus \{x, x', y, z\}$ and pick vertices $v\in N(y, u)$ and $w\in N(z, u)$, distinct from the existing vertices.
Note that both $\{x, y, z, u, v, w\}$ and $\{x', y, z, u, v, w\}$ span copies of $C_6$ in $H$ and thus $\{y, z, u, v, w\}$ is a reachable 5-set for $x$ and $x'$.
Since the number of choices for the reachable 5-sets is at least
\[
\frac 1{5!} \sqrt{\beta} n^2 (n-4) (b n -5) (b n - 6) \ge \beta n^5,
\]
because $\beta \ll b$ and $n$ is large enough, we have that $x'\in \tilde{N}_{\beta, 1}(x)$.

Note that $\delta_{2}(H)\ge b n$ implies that $\delta_{1}(H)\ge b \binom{n-1}{2}$. By double counting, we have
\[
 |N(x)|\delta_{2}(H) \le \sum_{S\in N(x)} \deg(S) < |\tilde{N}_{\beta, 1}(x)|\cdot |N(x)|+n\cdot \sqrt{\beta} {n}^{2}.
\]
Thus, $|\tilde{N}_{\beta, 1}(x)|> \delta_{2}(H) - \frac{\sqrt{\beta} n^3}{|N(x)|}\ge (b - \sqrt[3]\beta) n$ as $|N(x)|\ge \delta_1(H)\ge b \binom{n-1}{2}$.
\end{proof}

Propositions~\ref{prop:21} and \ref{prop:Nv} give the following corollary.

\begin{corollary}\label{coro:21}
For $0<\beta \ll b$ and integers $i_0'>i_0$, there exists $\beta'>0$ such that the following holds for sufficiently large $n$. 
Given an $n$-vertex $3$-graph $H$ with $\delta_2(H)\ge b n$. 
If $x,y\in V(H)$ are $(\beta,i_0)$-reachable in $H$, then $x,y$ are $(\beta',i_0')$-reachable in $H$.
\end{corollary}

The following lemma gives a useful partition of $V(H)$.
For its proof (in a more general form), see \cite{HZ_cycpac} (similar proofs can be found in \cite{Han14_poly, HLTZ_K4, HZZ_tiling}).

\begin{lemma}\label{lem:partition}
Given $0< \gamma \ll 1$, there exists $\beta>0$ such that the following holds for sufficiently large $n$. Let $H$ be an n-vertex 3-graph with $\delta_2(H) \ge (\frac{1}{3}-\gamma){n}$. Then there is a partition $\cP$ of $V(H)$ into $V_1, \dots, V_r$ with $r \le 3$ such that for any $i \in [r]$, $|V_i| \ge (\frac13 - 2\r) n$ and $V_i$ is $(\beta, 4)$-closed in $H$.
\end{lemma}

We apply Lemma \ref{lem:partition} on $H$ and get the partition $\cP=\{V_1,\dots, V_r\}$ such that each part is closed. 
Next lemma says that if $I_{\cP, C}^{\mu} (H)$ contains two 6-vectors whose difference equals a transferral $\bfu_i - \bfu_j$ for distinct $i,j\in [r]$, then $V_i\cup V_j$ is closed.
Note that our assumption here is stronger than assuming that $L_{\cP, C}^{\mu}(H)$ contains a transferral.

\begin{lemma}\label{lem:trans}
Given $\beta, \mu, b, r, c>0$, there exists $\beta'>0$ such that the following holds for sufficiently large $n$.
Let $H$ be an $n$-vertex 3-graph with $\delta_2(H)\ge b n$.
Let $\cP=\{V_1, \dots, V_r\}$ be a partition of $V(H)$ such that for any $i \in [r]$, $V_i$ is $(\beta, c)$-closed in $H$.
For distinct $i,j\in [r]$, $V_i\cup V_j$ is $(\beta', 7c+1)$-closed in $H$ if both 6-vectors
$(b_1,\dots, b_r), (b_1, \dots, b_r) + \bfu_i - \bfu_j \in I_{\cP, C}^{\mu} (H)$.
\end{lemma}

\begin{proof}
Without loss of generality, assume $i=1$ and $j=2$.
It suffices to show that every $x\in V_1$ and $y\in V_2$ are $(\beta'', 7c+1)$-reachable for some $\beta''>0$.
Indeed, since both $V_1$ and $V_2$ are $(\beta, c)$-closed in $H$. 
By Corollary~\ref{coro:21}, there exists $\beta'''$ such that they are $(\beta''', 7c+1)$-closed in $H$.
Then $V_1\cup V_2$ is $(\beta', 7c+1)$-closed in $H$ by letting $\beta' = \min \{\beta'', \beta'''\}$.

First, we pick a copy $F_1$ of $C_6$ with index vector $(b_1,\dots, b_r)$ and a copy $F_2$ of $C_6$ of index vector $(b_1+1, b_2-1, \dots, b_r)$ such that $F_1$ and $F_2$ are vertex disjoint and do not contain $x$ or $y$.
By the assumption, there are at least $\mu n^6 - 8 n^5 \ge \mu n^6/2$ choices for each of $F_1$ and $F_2$.
Let $x'\in V(F_2)\cap V_1$ and $y'\in V(F_1)\cap V_2$.
We name the other vertices as $V(F_1)\setminus \{y'\} = \{v_1,\dots, v_5\}$ and $V(F_2)\setminus \{x'\} = \{u_1,\dots, u_5\}$ such that for all $i\in [5]$, $u_i$ and $v_i$ belong to the same part of $\cP$, and thus they are $(\beta, c)$-reachable.
Next, we pick reachable $(6c-1)$-sets $S_x$ for $x$ and $x'$, $S_y$ for $y$ and $y'$, and for $i\in [5]$, we pick reachable $(6c-1)$-sets $S_i$ for $u_i$ and $v_i$ such that all these $(6c-1)$-sets are vertex disjoint and they contain no vertex in $\{x, y\}\cup V(F_1)\cup V(F_2)$.
Note that $S = V(F_1)\cup V(F_2)\cup (S_1\cup \cdots \cup S_5)\cup S_x\cup S_y$ is a reachable $(42c+5)$-set for $x$ and $y$. 
Indeed, $H[S\cup \{x\}]$ has a $C_6$-factor because by definition, $F_2$ is a copy of $C_6$ and, all of $H[S_x\cup \{x\}]$, $H[S_y\cup \{y'\}]$ and $H[S_i\cup \{v_i\}]$, $i\in [5]$ have $C_6$-factors.
Also, $H[S\cup \{y\}]$ has a $C_6$-factor because by definition, $F_1$ is a copy of $C_6$ and, all of $H[S_x\cup \{x'\}]$, $H[S_y\cup \{y\}]$ and $H[S_i\cup \{u_i\}]$, $i\in [5]$ have $C_6$-factors.

Note that for each of $S_1, \dots, S_5, S_x, S_y$, there are at least $\beta n^{6c-1} - (42c+5) n^{6c-2} \ge \beta n^{6c-1}/2$ choices for it.
In total, there are at least
\[
\frac1{(42c+5)!} \left( \frac{\mu n^6}2 \right)^2 \left( \frac{\beta n^{6c-1}}2 \right)^7 = \beta'' n^{42c+5},
\]
choices for $S$, where $\beta'' = \frac1{512(42c+5)!}\mu^2 \beta^7$.
So $x$ and $y$ are $(\beta'', 7c+1)$-reachable.
\end{proof}

Our next lemma is one of the key steps in proving Lemma~\ref{lem:absorb}.
Its proof is somehow long and we postpone it to the end of this section.

\begin{lemma}\label{lem:even}
Let $r=2,3$.
Suppose 
\[
0<1/n\ll \mu\ll \r \ll 1
\]
and let $H$ be an $n$-vertex 3-graph with $\delta_2(H) \ge (1/3-\gamma){n}$. Moreover, let $\cP=\{V_1, \dots, V_r\}$ be a partition of $V(H)$ with $|V_i| \ge n/3-2\gamma n$ for $i\in [r]$. 
Then one of the following holds.
\begin{itemize}
\item[(i)] There exist a 6-vector $\bfv$ and distinct $i,j\in [r]$ such that $\bfv, \bfv+\bfu_i-\bfu_j \in I_{\cP,C}^{\mu}(H)$.
\item[(ii)] 
All 6-vectors with all coordinates even are in $I_{\cP,C}^{\mu}(H)$. 
Moreover, if $r=2$, then $(1,2), (2,1)\in I_{\cP}^{\mu}(H)$.
\end{itemize}
\end{lemma}

The following lemma extends \cite[Proposition 8.2]{My14} -- it works under a slightly lower codegree and a slightly more unbalanced bipartition.
The proof is similar to the one of \cite[Proposition 8.2]{My14}, except that we use Lemma~\ref{lem:even}.

\begin{lemma}\label{lem:oddcopy}
Given $0< \gamma \ll 1$, the following holds for sufficiently large $n$. Let $H=(V, E)$ be an n-vertex 3-graph with $\delta_2(H) \ge (\frac{1}{3}-\gamma){n}$. Suppose $A\cup B$ is a bipartition of $V$ such that $|A|, |B|\ge n/3 - 2\r n$, then there is a copy of $C_6$ that intersects $A$ at an odd number of vertices.
\end{lemma}

\begin{proof}
Let $0<1/n \ll \mu \ll \r$.
Suppose for a contradiction that no such copy of $C_6$ exists.
Without loss of generality, assume that $|A| \le n/2$.
Note that $(2,1)\in I_{\cP}^\mu(H)$ by Lemma~\ref{lem:even} with $r=2$.
Indeed, otherwise, Lemma~\ref{lem:even}(i) holds and exactly one of the two robust $C_6$-vectors has odd coordinates, implying the existence of a desired copy of the lemma, a contradiction.

Color the edges of the complete graph $K[A]$ as follows.
In fact, we color $xy$ red if there are at least 3 vertices $w\in B$ with $\{x, y, w\}\in E$, and we color $xy$ blue if there are at least 6 vertices $w\in A$ such that $\{x, y, w\}\in E$.
So every edge $xy$ receives at least one color.
Since any pair $xy$ lies in at most $n$ edges, we find that there are at least $(\mu n^3 - 2n^2)/n \ge \mu n^2/2$ red edges of $K[A]$.

Observe that no triangle in $K[A]$ has three red edges.
Indeed, if $xyz$ is such a triangle then we may choose distinct $w_1, w_2, w_3\in B$ such that $\{x, y, w_1\}, \{x, z, w_2\}, \{y, z, w_3\}$ are each edges of $H$, thus forming a copy of $C_6$ with index vector $(3,3)$.
Similarly, no triangle in $K[A]$ has two blue edges and one red edge, as then we can find a copy of $C_6$ with index vector $(5,1)$.
Now, choose any vertex $x\in A$ which lies in a red edge, and define $A_1=\{y\in A\setminus \{x\}: xy \text{ is red}\}$ and $A_2: = A\setminus A_1$.
So $A_1$ and $A_2$ partition $A$, and by our previous observations no edge of $K[A_1]$ or $K[A_2]$ is red.
So all edges of $K[A_1]$ and $K[A_2]$ are blue and not red; it follows that every edge $yz$ with $y\in A_1$ and $z\in A_2$ is red and not blue (so in fact every edge of $K[A]$ has only one color).
Moreover, the red edges of $K[A]$ form a complete bipartite subgraph of $K[A]$ with vertex classes $A_1$ and $A_2$.
Since the number of red edges of $K[A]$ is at least $\mu n^2/2$ it follows that $|A_1|, |A_2| \ge \mu n/2$.
Without loss of generality we may assume that $|A_1|\le |A_2|$, so $|A_1|\le n/4$.

Let $y,z\in A_1$. There are at least $\delta_2(H)\ge (\frac{1}{3}-\gamma){n}$ vertices $w$ such that $\{w, y, z\}\in E$. At most $n/4$ of these vertices $w$ lie in $A_1$, and since $yz$ is not red at most 2 of these vertices $w$ lie in $B$.
So there are at least $\mu n$ vertices $w\in A_2$ such that $\{w, y, z\}\in E$; summing over all pairs $y, z\in A_1$ we find that there are at least $\binom{|A_1|}2 \mu n \ge \mu^3 n^3/9$ edges of $H$ with two vertices in $A_1$ and one vertex in $A_2$.
Since there are $|A_1| |A_2|\le n^2$ pairs $yz$ with $y\in A_1$ and $z\in A_2$, we deduce that some such pair $yz$ lies in at least $\mu^3 n /9 \ge 6$ such edges of $H$.
But then $yz$ is blue, a contradiction.
\end{proof}

\subsection{Proof of Lemma~\ref{lem:absorb}}

We call an $m$-set $A$ an \emph{absorbing $m$-set} for a $6$-set $S$ if $A\cap S=\emptyset$ and both $H[A]$ and $H[A\cup S]$ contain $C_6$-factors. 
Denote by $\A^m(S)$ the set of all absorbing $m$-sets for $S$.
Now we are ready to prove Lemma \ref{lem:absorb}.

\begin{proof}[Proof of Lemma~\ref{lem:absorb}.]
Suppose $0<1/n \ll \{\beta, \mu\} \ll \r, 1/t$.
Suppose $H$ is an $n$-vertex 3-graph with $\delta_2(H) \ge (\frac{1}{3}-\gamma){n}$. 
Applying Lemma~\ref{lem:partition} on $H$ gives a partition $\cP'$ of $V(H)$ into $V_1', \dots, V_{r'}'$ with $r' \le 3$ such that for any $i \in [r']$, $|V_i'| \ge (\frac13 - 2\r) n$ and $V_i'$ is $(\beta', 4)$-closed in $H$ for some $\beta \ll \beta' \ll \r$.
By Lemma~\ref{lem:trans}, we combine the parts $V_i', V_j'$ if there exist 6-vector $\bfv$ and distinct $i,j\in [r']$ such that $\bfv, \bfv+ \bfu_i - \bfu_j \in I_{\cP', C}^{\mu}(H)$.
We greedily combine the parts (at most twice) until there is no such $\mu$-robust 6-vectors $\bfv$ and $\bfv+\bfu_i - \bfu_j$.
Let $\cP=\{V_1,\dots, V_r\}$ be the resulting partition with $r\le 3$. 
By Corollary~\ref{coro:21}, we may assume that for any $i \in [r]$, $V_i$ is $(\beta, t)$-closed in $H$ for some $t\le 204$.
Moreover, by Lemma~\ref{lem:even}, we may assume that all 6-vectors with all coordinates even are in $I_{\cP, C}^\mu(H)$.

Let $\F_0 = \emptyset$ if $r=1$.
If $r=2$, then we apply Lemma~\ref{lem:oddcopy} on $\{V_1, V_2\}$ and get a copy $F_0$ of $C_6$ that intersects both parts of $\cP$ at an odd number of vertices.
Let $\F_0 = \{F_0\}$.
If $r=3$, then we apply Lemma~\ref{lem:oddcopy} on $\{V_1, V_2\cup V_3\}$ and get a copy $F_1$ of $C_6$ that intersects $V_1$ and $V_i$ at an odd number of vertices, where $\{i, j\}=\{2,3\}$.
Then we apply Lemma~\ref{lem:oddcopy} on $\{V_j\setminus V(F_1), (V_1\cup V_i)\setminus V(F_1)\}$ and get a copy $F_2$ of $C_6$ that intersects $V_j$ and one of $V_1$ and $V_i$ at an odd number of vertices. 
So $\bfi_{\cP}(F_1) \pmod 2$ and $\bfi_{\cP}(F_2) \pmod 2$ are two distinct vectors from $(1,1,0)$, $(0,1,1)$ and $(1,0,1)$.
Let $\F_0=\{F_1, F_2\}$.

Let $m=36t$, $\r_1= \mu \beta^{6}/128$ and $\a = \r_1^2$.

\begin{claim}\label{clm:abs}
Any $6$-set $S$ with all coordinates even satisfies that $|\A^m(S)|\ge \r_1 n^{m}$.
\end{claim}

\begin{proof}
For a $6$-set $S=\{y_1,\dots, y_6\}$ with all coordinates even, we construct absorbing $m$-sets for $S$ as follows. 
We first fix a copy $F$ of $C_6$ on $\{x_1, \dots, x_6\}$ in $H$ such that $\bfi_{\cP}(F)=\bfi_{\cP}(S)$ and $F\cap S=\emptyset$, for which we have at least $\mu n^6 - 6 n^{5} >{\mu} n^6/2$ choices. 
Without loss of generality, we may assume that for all $i\in [6]$, $x_i, y_i$ are in the same part of $\cP$. Since $x_i$ and $y_i$ are $(\beta, t)$-reachable, there are at least $\beta n^{6t-1}$ $(6t-1)$-sets $T_i$ such that both $H[T_i\cup \{x_i\}]$ and $H[T_i\cup \{y_i\}]$ have $C_6$-factors. We pick disjoint reachable $(6t-1)$-sets for each $x_i, y_i$, $i\in [6]$ greedily, while avoiding the existing vertices. Since the number of existing vertices is at most $m$, there are at least ${\beta} n^{6t-1}/2$ choices for each such $(6t-1)$-set.
Note that each $F\cup T_1\cup \cdots \cup T_{6}$ is an absorbing set for $S$. Indeed, first, it contains a $C_6$-factor because each $T_i\cup \{x_i\}$ for $i\in [6]$ spans $t$ disjoint copies of $C_6$. Second, $H[F\cup T_1\cup \cdots \cup T_{6}\cup S]$ also contains a $C_6$-factor because $F$ is a copy of $C_6$ and each $T_i\cup \{y_i\}$ for $i\in [6]$ spans $t$ disjoint copies of $C_6$.  So we get at least $\r_1 n^{m}$ absorbing $m$-sets for $S$.
\end{proof}

Now we build a family $\F_1$ of $m$-sets by probabilistic arguments. Choose a family $\F$ of $m$-sets in $H$ by selecting each of the $\binom nm$ possible $m$-sets independently with probability $p= \r_1 n^{1-m}$. Then by Chernoff's bound, with probability $1-o(1)$ as $n \rightarrow \infty$, the family $\F$ satisfies the following properties:
\begin{align}
|\F|\leq 2p\binom nm\leq {\r_1 n}\, \text{ and }\, |\A^m(S)\cap \F|\geq \frac{p|\A^m(S)|}2\geq \frac{\r_1^{2} n}{2},  \label{expected}
\end{align}
for all 6-sets $S$ with all coordinates even.
Furthermore, the expected number of pairs of $m$-sets in $\F$ that are intersecting is at most
\begin{align*}
\binom nm\cdot m \cdot \binom n{m-1} \cdot p^2\leq \frac{\r_1^2n}{8}.
\end{align*}
Thus, by using Markov's inequality, we derive that with probability at least $1/2$,
\begin{align}
\F \text{ contains at most } \frac{\r_1 ^{2}n}{4} \text{ intersecting pairs of $m$-sets.}  \label{intersecting}
\end{align}
Hence, there exists a family $\F$ with the properties in \eqref{expected} and \eqref{intersecting}. By deleting one member of each intersecting pair, the $m$-sets intersecting $V(\F_0)$, and the $m$-sets that are not absorbing sets for any $6$-set $S\subseteq V$, we get a subfamily $\F_1$ consisting of pairwise disjoint $m$-sets.
Let $W=V(\F_1)\cup V(\F_0)$ and thus $|W|\le m|\F|+12 < m\r_1 n + 12 < \r n$.
Since every $m$-set in $\F_1$ is an absorbing $m$-set for some $6$-set $S$ and every element of $\F_0$ is a copy of $C_6$, $H[W]$ has a $C_6$-factor.
For any $6$-set $S$ with all coordinates even, by \eqref{expected} and \eqref{intersecting} above we have
\begin{align}
|\A^m(S)\cap \F_1|\geq \frac{\r_1^{2} n}{2}-\frac{\r_1^2 n}{4} -|V(\F_0)| \ge \frac{\r_1^{2} n}{4} - 12\label{eq:AS}.
\end{align}

Now fix any set $U\subseteq V\setminus W$ of size $|U|\leq \a n$ and $|U|\in 6\mathbb{Z}$.
We claim that there exists $\F'\subseteq \F_0$ such that $U\cup V(\F')$ can be partitioned into at most $\a n/6+2$ 6-sets with all coordinates even.
Indeed, first observe that a set $U'$ with $|U'|\in 6\mathbb{Z}$ can be partitioned into 6-sets with all coordinates even if and only if all coordinates of $\bfi_{\cP}(U')$ are even.
If $r=1$, then $\bfi_{\cP}(U) = (|U|)$ is even.
If $r=2$, then either $\bfi_{\cP}(U)$ or $\bfi_{\cP}(U\cup V(F_0))$ has all coordinates even.
Otherwise $r=3$. If not all coordinates of $\bfi_{\cP}(U)$ are even, then $\bfi_{\cP}(U) \pmod 2 \in \{(1,1,0), (1,0,1), (0,1,1)\}$.
Thus, exactly one of $\bfi_{\cP}(U\cup V(F_1))$, $\bfi_{\cP}(U\cup V(F_2))$ and $\bfi_{\cP}(U\cup V(F_1\cup F_2))$ have all coordinates even. 
So the claim holds.
Since each $6$-set has all coordinates even, by \eqref{eq:AS} and $\frac{\a n}6+2 \leq \frac{\r_1^2 n}{4} - 12$, they can be greedily absorbed by $m$-sets in $\F_1$. Hence, $H[U\cup W]$ contains a $C_6$-factor.
\end{proof}

\subsection{Proof of Lemma~\ref{lem:even}}

We first collect some useful simple facts on graphs. 

\begin{fact}\label{fact:triangle}
Fix $0<\r, \r' <1$. and let $G$ be a graph on $V$. 
\begin{enumerate}[$(i)$]
\item If $|E(G)| \ge (1-\r){|V|\choose 2}$, then the number of triangles in $G$ is at least $(1-3\r) {|V|\choose 3}$.
\item If $G=(V_1, V_2, V_3, E)$ is tripartite and we have $e(V_i, V_j) \ge (1-\r)|V_i| |V_j|$ for distinct $i,j\in [3]$, then the number of triangles in $G$ is at least $(1-3\r) |V_1| |V_2| |V_3|$.
\item Suppose $V=V_1\cup V_2$ for some $V_1\cap V_2=\emptyset$ and $|V_1|\ge \r'/\r$. If $e(V_1) \ge (1-\r)\binom{|V_1|}2$ and $e(V_1, V_2) \ge \r'|V_1| |V_2|$, then the number of triangles in $G$ with two vertices in $V_1$ and one vertex in $V_2$ is at least $(\r'^2 - 2\r) \binom{|V_1|}2 |V_2|$.
\end{enumerate}
\end{fact}

\begin{proof}
We only prove (iii) because the first two are immediate by counting the triples containing non-edges.
Since $e(V_1, V_2) \ge \r'|V_1| |V_2|$, the number of copies of $P_3$ centred at some vertex in $V_2$ is at least
\[
\sum_{v \in V_2} {\deg(v)\choose 2} \ge \frac{1}{|V_2|}\frac{(\r'|V_1| |V_2|)^2}{2}-\frac{\r'|V_1| |V_2|}{2} \ge (\r'^2-\r) \binom{|V_1|}{2} |V_2|,
\]
where we used that $|V_1| \ge \r'/\r$.
Note that among these copies of $P_3$, at most $\r \binom{|V_1|}2 |V_2|$ of them miss the edge in $V_1$, and thus the result follows.
\end{proof}

We will also use the following simple fact in the proof of Lemma~\ref{lem:even}.

\begin{fact}\label{fact:lattco}
Given an integer $r\ge 1$ and $\mu \ll \delta, 1/r$, suppose $H$ is an $n$-vertex 3-graph with $\delta_2(H) \ge \delta n$ where $n$ is large enough. Let $\cP=\{V_1, \dots, V_r\}$ be a partition of $V(H)$ with $|V_i|\ge \delta n$. For every 2-vector $\bfv\in \mathbb{Z}^r$, there exists $i\in [r]$ such that $\bfv + \bfu_i\in I_{\cP}^\mu (H)$.
\end{fact}

\begin{proof}
Fix any 2-vector $\bfv$, the number of pairs $p$ in $V(H)$ with respect to this index vector is at least $\binom{\delta n}2$. Thus the number of hyperedges in $H$ containing these pairs is at least $\frac13\delta n \binom{\delta n}2 \ge \binom{\delta n}3$. Since $\mu \ll \delta$, we have $r\mu n^3 < \binom{\delta n}3$. By averaging, there must be an $i\in [r]$ such that at least $\mu n^3$ edges $e\in E(H)$ satisfy $\mathbf{i}_\cP(e)=\bfv + \bfu_i$, which shows that $\bfv + \bfu_i\in I_{\cP}^\mu (H)$.
\end{proof}

Here we state a simple counting result and omit its proof. 

\begin{proposition}\label{supersaturation}
For $1/n\ll \mu$, every 3-graph $H$ on $n$ vertices with at least $\mu n^3$ edges contains at least $\mu^{8} n^{6}/2$ copies of $K_3^3(2)$.
\end{proposition}

Given a partition $\cP$, $0<\mu<1$ and a $\mu$-robust edge-vector $\mathbf{i}$, by Proposition \ref{supersaturation}, the edges with index vector $\mathbf{i}$ form at least $\mu' n^6$ copies of $C_6$ with index vector $2\bfi$, where $\mu'  = \mu^8/2$, i.e., $2\bfi \in I_{\cP, C}^{\mu'}(H)$. 
For example, given $r=2$ and $(1,2)\in I_{\cP}^{\mu}(H)$, then $(2,4)\in I_{\cP, C}^{\mu'}(H)$.

\begin{proof}[Proof of Lemma~\ref{lem:even}.]
Let $0<1/n\ll \mu\ll \eta \ll \r \ll 1$.
Note that by Proposition~\ref{supersaturation} and $\mu \ll \eta$, instead of assuming that (ii) does not hold, we may assume that there is some 3-vector $\bfv$ such that $\bfv\notin I_{\cP}^\eta (H)$ -- otherwise (ii) holds.
Then it suffices to show that either (i) holds, or $2\bfv \in I_{\cP,C}^{\mu}(H)$.
(The `moreover' part of (ii) will be explained during the proof.)

We will use the following notion in the proof.
Suppose that $\bfv\notin I_{\cP}^\eta (H)$, where $\bfv = \bfu_{i} + \bfu_{j} + \bfu_k$ is a 3-vector for some multi-set $\{i, j, k\}$, $i,j,k\in [r]$.
Let $\bfv' = \bfu_i + \bfu_j$ be a 2-vector.
Then, for each pair $S$ of vertices such that $\bfi_{\cP}(S)=\bfv'$, we call it \emph{bad} if $\deg(S, V_k)\ge \r n$ (otherwise \emph{good}).
Thus, since $\eta \ll \r$ and $|V_1|, |V_2|, |V_3|\ge n/3 - 2\r n$, the number of bad pairs with index vector $\bfv'$ is at most
\[
3\eta n^3/ (\r n) = \frac{3\eta}{\r} n^2 \le \r \text{vol}(V_i, V_j),
\]
where $\text{vol}(V_i, V_j)$ stands for the number of pairs $uv$ such that $u\in V_i$ and $v\in V_j$, i.e. $\text{vol}(V_i, V_j)=|V_i||V_j|$ if $i\neq j$, $\text{vol}(V_i, V_j) = \binom{|V_i|}2$ if $i=j$.
Note that since $\bfv'$ may not be unique, so we may have defined more than one `goodness'.
In each (sub)case of the proof, we will consider the triples with index vector $\bfv$ such that all three pairs in the triple are good (possibly with further restrictions).

\medskip
\noindent\textbf{Case 1. $r=2$.}
By symmetry, we only need to deal with two subcases, $(3,0)\notin I_\cP^{\eta}(H)$ or $(2,1)\notin I_\cP^{\eta}(H)$.

First assume that $(3,0)\notin I_\cP^{\eta}(H)$. 
Note that by Fact~\ref{fact:lattco}, $(3,0)\notin I_\cP^{\eta}(H)$ implies that $(2,1)\in I_\cP^{\eta}(H)$.
Thus $(4,2)\in I_{\cP,C}^{\mu}(H)$ by Proposition \ref{supersaturation}.
Also, note that the number of bad pairs in $V_1$ is at most $\r \binom{|V_1|}2$. By Fact \ref{fact:triangle}(i), there are at least $(1-3\r) \binom{|V_1|}3$ triples in $\binom{V_1}3$ of which all pairs are good. For each such triple, we pick distinct neighbors of the three pairs in $V_2$ and get a copy of $C_6$ with index vector $(3,3)$. 
There are at least 
\begin{equation}\label{eq:case11}
\frac1{6!}(1-3\r) \binom{|V_1|}3 (\delta_2(H) - \r n) \left(\delta_2(H) - \r n-1\right) \left(\delta_2(H) - \r n-2\right) \ge  \mu n^6
\end{equation}
such copies of $C_6$ with index vector $(3,3)$ by $\mu\ll 1$ and $\delta_2(H)\ge n/3 - \r n$. 
This means that $(3,3) \in I_{\cP,C}^{\mu}(H)$. 
Since $(4,2), (3,3)\in I_{\cP,C}^{\mu}(H)$, (i) holds.

Now assume $(2,1)\notin I_\cP^{\eta}(H)$. 
By Fact~\ref{fact:lattco}, $(2,1)\notin I_\cP^{\eta}(H)$ implies that $(1,2)\in I_\cP^{\eta}(H)$.
Thus $(2,4)\in I_{\cP,C}^{\mu}(H)$ by Proposition \ref{supersaturation}.
Note that the number of bad pairs in $V_1\times V_2$ is at most $\r |V_1||V_2|$ and the number of bad pairs in $V_1$ is at most $\r {|V_1|\choose 2}$.
By applying Fact~\ref{fact:triangle}(iii) with $\r'=1-\r$, we see that the number of triples with index vector $(2,1)$ such that all pairs of the triple are good is at least $(1-4\r) \binom{|V_1|}2|V_2|$. For each such triple, we pick distinct neighbors in $V_2$ of the pairs in $V_1\times V_2$ and pick a neighbor in $V_1$ of the pair in $V_1$ and get a copy of $C_6$ with index vector $(3,3)$. There are at least 
\[
\frac1{6!}(1-4\r) \binom{|V_1|}2|V_2| \cdot  (\delta_2(H) - \r n) \left(\delta_2(H) - \r n-1\right) \left(\delta_2(H) - \r n\right) \ge  \mu n^6
\]
such copies of $C_6$ with index vector $(3,3)$.
This means that $(3,3) \in I_{\cP,C}^{\mu}(H)$. Together with $(2,4)\in I_{\cP,C}^{\mu}(H)$, (i) holds.
(By symmetry, this shows the `moreover' part of the lemma.)

\medskip
\noindent\textbf{Case 2. $r=3$.}
By symmetry, we only need to deal with three subcases, $(3,0,0)\notin I_\cP^{\eta}(H)$, $(2,1,0)\notin I_\cP^{\eta}(H)$ or $(1,1,1)\notin I_\cP^{\eta}(H)$.

First assume that $(2,1,0)\notin I_\cP^{\eta}(H)$. 
Note that the number of bad pairs in $V_1$ is at most $\r \binom{|V_1|}2$ and the number of bad pairs in $V_1\times V_2$ is at most $\r |V_1| |V_2|$.
Also note that each good pair $S\in V_1\times V_2$ satisfies that $\deg(S, V_2\cup V_3)\ge \delta_2(H) - \r n\ge (1/3- 2\r)n$, which implies that $\deg(S, V_2)\ge n/7$ or $\deg(S, V_3)\ge n/7$.
Assume that there are at least $\frac{1-\r}2 |V_1| |V_2| \ge |V_1| |V_2|/3$ good pairs $S$ in $V_1\times V_2$ such that $\deg(S, V_2)\ge n/7$ (the other case will be quite similar).
This implies
\[
e(V_1, V_2, V_2) \ge \frac12 \frac{|V_1||V_2|}3 \frac n7 \ge \eta n^3
\]
by $\eta\ll 1$. Thus, $(1,2,0)\in I_{\cP}^{\eta}(H)$ and $(2,4,0)\in I_{\cP, C}^{\mu}(H)$ by Proposition~\ref{supersaturation}.

By applying Fact~\ref{fact:triangle}(iii) with $\r' = 1/3$, the number of triples $\{x,y,z\}$ with 
$x,y\in V_1$, $z\in V_2$ such that $xy$ is good, $\deg(xz, V_2)\ge n/7$ and $\deg(yz, V_2)\ge n/7$ is at least $(1/9 - 2\r) \binom{|V_1|}2|V_2|$. For each such triple, we pick distinct neighbors of $xz, yz$ in $V_2$ and pick a neighbor of $xy$ in $V_1\cup V_3$ and get a copy of $C_6$ with index vector $(3,3,0)$ or $(2,3,1)$. There are at least 
\[
\frac1{6!} (1/9 - 2\r) \binom{|V_1|}2|V_2| \cdot  \frac n7 \left(\frac n7 -1\right) \left(\delta_2(H) - \r n\right) \ge  2\mu n^6
\]
such copies of $C_6$ with index vector $(3,3,0)$ or $(2,3,1)$.
This means that $(3,3,0)$ or $(2,3,1) \in I_{\cP,C}^{\mu}(H)$. Together with $(2,4,0)\in I_{\cP,C}^{\mu}(H)$, (i) holds.

Second assume that $(3,0,0)\notin I_\cP^{\eta}(H)$. 
By the last subcase, we may assume that both $(2,1,0)\in I_\cP^{\eta}(H)$ and $(2,0,1)\in I_\cP^{\eta}(H)$. Then we have $(4,2,0), (4,0,2)\in I_{\cP,C}^{\mu}(H)$ by Proposition \ref{supersaturation}. We treat $V_2 \cup V_3$ as one part and use the proof of the first part in Case 1. 
Note that we can strengthen the consequence of \eqref{eq:case11} to $4\mu n^6$, which allows us to conclude that at least one of $(3,3,0), (3,2,1), (3,1,2), (3,0,3)$ is in $I_{\cP, C}^{\mu}(H)$. 
If $(3,3,0)$ or  $(3,2,1)$ is in $I_{\cP, C}^{\mu}(H)$, then it together with $(4,2,0)$ implies (i). If $(3,1,2)$ or  $(3,0,3)$ is in $I_{\cP, C}^{\mu}(H)$, then it together with $(4,0,2)$ implies (i).
So we are done.

Finally, assume that $(1,1,1)\notin I_\cP^{\eta}(H)$. 
Note that for distinct $i,j\in [3]$, the number of bad pairs in $V_i\times V_j$ is at most $\r |V_i||V_j|$.
By Fact~\ref{fact:triangle}(ii), the number of triples with index vector $(1,1,1)$ such that all pairs are good is at least $(1-3\r) |V_1| |V_2| |V_3|$. For each such triple, we pick distinct neighbors in $V_i\cup V_j$ of the pair in $V_i\times V_j$ for all distinct $i,j\in [3]$ and get a copy of $C_6$. There are at least 
\[
\frac1{6!}(1-3\r) |V_1| |V_2| |V_3| \cdot  (\delta_2(H) - \r n) \left(\delta_2(H) - \r n-1\right) \left(\delta_2(H) - \r n - 2\right) \ge  7\eta n^6
\]
such copies of $C_6$.
Observe that in each such copy of $C_6$, the triple has index vector $(1,1,1)$ and the three new vertices cannot fall into the same part of $\cP$.
So the index vector of such copy of $C_6$ is either $(2,2,2)$ or a permutation of $(3,2,1)$.
We first assume that there are at least $\eta n^6$ such copies of $C_6$ with index vector $(3,2,1)$.
Observe that each such copy of $C_6$ with index vector $(3,2,1)$ contains an edge of index vector $(2,1,0)$ (in fact, the index vectors of the three edges must be exactly $(2,1,0)$, $(2,0,1)$ and $(0,2,1)$).
Thus, we see at least $\eta n^6/n^3= \eta n^3$ edges of index vector $(2,1,0)$, i.e., $(2,1,0)\in I_{\cP}^{\eta}(H)$.
By Proposition~\ref{supersaturation}, this implies that $(4,2,0)\in I_{\cP,C}^{\mu}(H)$.
Together with $(3,2,1)\in I_{\cP,C}^{\mu}(H)$, (i) holds.
By symmetry, the only case left is that $(2,2,2)=2(1,1,1) \in I_{\cP,C}^{\eta}(H)\subseteq I_{\cP,C}^{\mu}(H)$.
Then (ii) holds and we are done.
\end{proof}

\section{Almost perfect $K_k^k(h)$-tiling}

\subsection{The Weak Regularity Lemma}

We first introduce the \emph{Weak Regularity Lemma}, which is a straightforward extension of Szemer\'edi's regularity lemma for graphs \cite{Sze}.

Let $H = (V, E)$ be a $k$-graph and let $A_1, \dots, A_k$ be mutually disjoint non-empty subsets of $V$. We define 
the density of $H$ with respect to ($A_1, \dots, A_k$) as
\[
d(A_1,\dots, A_k) = \frac{e(A_1, \dots, A_k)}{|A_1| \cdots|A_k|}.
\]
We say a $k$-tuple ($V_1, \dots, V_k$) of mutually disjoint subsets $V_1, \dots, V_k\subseteq V$ is \emph{$(\e, d)$-regular}, for $\e>0$ and $d\ge 0$, if
\[
|d(A_1, \dots, A_k) - d|\le \e
\]
for all $k$-tuples of subsets $A_i\subseteq V_i$, $i\in [k]$, satisfying $|A_i|\ge \e |V_i|$. We say ($V_1, \dots, V_k$) is \emph{$\e$-regular} if it is $(\e, d)$-regular for some $d\ge 0$. 

\begin{theorem}[Weak Regularity Lemma]
\label{thmReg}
Given $t_0\ge 0$ and $\e>0$, there exist $T_0 = T_0(t_0, \e)$ and $n_0 = n_0(t_0,\e)$ so that for every $k$-graph $H = (V, E)$ on $n>n_0$ vertices, there exists a partition $V = V_0 \cup V_1 \cup \cdots \cup V_t$ such that
\begin{enumerate}
\item[(i)] $t_0\le t\le T_0$,
\item[(ii)] $|V_1| = |V_2| = \cdots = |V_t|$ and $|V_0|\le \e n$,
\item[(iii)] for all but at most $\e \binom tk$ $k$-subsets $\{i_1,\dots, i_k\} \subset [t]$, the $k$-tuple $(V_{i_1}, \dots, V_{i_k})$ is $\e$-regular.
\end{enumerate}
\end{theorem}

The partition given in Theorem \ref{thmReg} is called an \emph{$\e$-regular partition} of $H$. Given an $\e$-regular partition of $H$ and $d\ge 0$, we refer to $ V_i, i\in [t]$ as \emph{clusters} and define the \emph{cluster hypergraph} $\K = \K(\e,d)$ with vertex set $[t]$ and $\{i_1,\dots,i_k\}\subset [t]$ is an edge if and only if $(V_{i_1}, \dots, V_{i_k})$ is $\e$-regular and $d(V_{i_1}, \dots, V_{i_k}) \ge d$.

We combine Theorem \ref{thmReg} and \cite[Proposition 16]{HS} into the following corollary, which shows that the cluster hypergraph almost inherits the minimum degree of the original hypergraph.
Its proof is standard and similar as the one of \cite[Proposition 16]{HS} so we omit it.

\begin{corollary} [\cite{HS}]
\label{prop16}
Given $c, \e, d>0$ and $t_0$, there exist $T_0$ and $n_0$ such that the following holds. Let $H$ be a $k$-graph on $n>n_0$ vertices with $\delta_{k-1}(H)\ge c n$. Then $H$ has an $\e$-regular partition $V_0\cup V_1 \cup \cdots \cup V_t$ with $t_0\le t\le T_0$, and in the cluster hypergraph $\K = \K(\e,d)$, all but at most $\sqrt \e t^{k-1}$ $(k-1)$-subsets $S$ of $[t]$ satisfy $\deg_{\K}(S)\ge (c - d - \sqrt{\e})t - (k-1)$.
\end{corollary}

\subsection{The Proof of Lemma \ref{lem:tiling}}
The following lemma provides an almost perfect matching under the defect minimum codegree as in Corollary \ref{prop16}. Its proof is similar to the proof of \cite[Lemma 1.7]{Han14_mat}.

\begin{lemma} [Almost perfect matching] \label{almost}
For any integer $k\ge 3$ and $0<\e \ll \a, \r$ the following holds for sufficiently large $n$. Let $H=(V, E)$ be an $n$-vertex $k$-graph such that all but at most $\e n^{k-1}$ $(k-1)$-sets $S\subseteq V$ satisfy that $\deg(S)\ge n/k-\r n$. 
If $H$ is not $\r$-extremal, then $H$ contains a matching that covers all but at most $\a n$ vertices of $V$.
\end{lemma}

\begin{proof}
Let $M=\{e_1, e_2, \dots, e_m\}$ be a maximum matching of size $m$ in $H$. Let $V' = V(M)$ and let $U = V\setminus V'$. We assume that $H$ is not $\r$-extremal and $|U|> \a n$. 
Note that $U$ is an independent set by the maximality of $M$. 

Let $t=\lceil k/\r \rceil$. 
We greedily pick disjoint $(k-1)$-sets $A_1, \dots, A_t$ in $U$ such that $\deg(A_{i}) \ge n/k - \r n$ for all $i\in [t]$.
This is possible since in each step, the number of $(k-1)$-sets that intersect the existing sets or have low degree is at most
\[
(k-1)t \cdot \binom{|U|}{k-2} + \e n^{k-1} \le \frac{k^3}{\r |U|} \binom{|U|}{k-1} + \frac{k! \e}{\a^{k-1}} \binom{|U|}{k-1} < \binom{|U|}{k-1},
\]
because $|U| > \a n > 2k^3/\r$ and $\e \ll \a$.
So we can pick the desired $(k-1)$-set.

Let $D$ be the set of vertices $v\in V'$ such that $\{v\}\cup A_{i}\in E$ for at least $k$ sets $A_{i}$, $i\in [t]$. 
We claim that $|e_i\cap D|\le 1$ for any $i\in [m]$. Indeed, otherwise, assume that $x, y\in e_i\cap D$. By the definition of $D$, we can  pick $A_i, A_j$ for some distinct $i, j\in [t]$ such that $\{x\}\cup A_i\in E$ and $\{y\}\cup A_j\in E$. We obtain a matching of size $m+1$ by replacing $e_i$ in $M$ by $\{x\}\cup A_i$ and $\{y\}\cup A_j$, contradicting the maximality of $M$.

We claim that $|D|\ge (\frac 1k -2\r)n$. Indeed, by the degree condition, we have
\[
t\left(\frac 1k -\r \right)n\le \sum_{i=1}^t \deg(A_{i})\le |D| t +n\cdot k,
\]
where we use the fact that $U$ is an independent set.
So we get
\[
|D|\ge \left(\frac 1k -\r \right)n - \frac{n k}{t} \ge \left(\frac 1k -2\r \right)n,
\]
where we use $t\ge k/\r$.

Let $V_D:=\bigcup\{e_i, e_i\cap D\neq \emptyset\}$. Note that $|V_D\setminus D| = (k-1)|D|\ge (k-1)(\frac 1k - 2\r)n = \frac{k-1}{k}n - 2\r(k-1)n$.
We observe that if $H[V_D\setminus D]$ spans no edge, then by adding $\lfloor \frac{k-1}kn \rfloor - |V_D\setminus D| \le 2\r (k-1) n$ vertices, we get a set of size $\lfloor \frac{k-1}kn \rfloor$ which spans at most
\[
2\r(k-1)n \binom{\frac{k-1}kn}{k-1} < \r n^k
\]
edges.
Since $H$ is not $\r$-extremal, $H[V_D\setminus D]$ contains at least one edge, denoted by $e_0$.
We assume that $e_0$ intersects $e_{i_1}, \dots, e_{i_l}$ in $M$ for some $2\le l\le k$. Suppose $\{v_{i_j}\}= e_{i_j}\cap D$ for all $j\in [l]$. By the definition of $D$, we can greedily pick $A_{i_1}, \dots, A_{i_l}$ such that  $\{v_{i_j}\}\cup A_{i_j} \in E$ for all $j\in [l]$. Let $M''$ be the matching obtained from replacing the edges $e_{i_1}, \dots, e_{i_l}$ by $e_0$ and $\{v_{i_j}\}\cup A_{i_j}$ for $j\in [l]$. Thus, $M''$ has $m+1$ edges, contradicting the maximality of $M$.
\end{proof}

Now we are ready to prove Lemma \ref{lem:tiling}.

\begin{proof}[Proof of Lemma~\ref{lem:tiling}.]
Fix integers $k, h$, $0<\e \ll \r,\a<1$. Let $n'$ be the constant returned from Lemma \ref{almost} with $0<\e \ll 2\r, \a$. 
Let $T_0$ be the constant returned from Corollary \ref{prop16} with $c=\frac 1{k} - \r$, $\e^2$, $d=\r/2$ and 
$t_0 > \max\{n', 4k/\r\}$.

Let $n$ be sufficiently large and let $H$ be a $k$-graph on $n$ vertices with $\delta_{k-1}(H) \ge (\frac 1{k} - \r) n$. Applying Corollary \ref{prop16} with the constants chosen above, we obtain an $\e^2$-regular partition and a cluster hypergraph $\K = \K(\e^2, d)$ on $[t]$ such that for all but at most $\e t^{k-1}$ $(k-1)$-sets $S\in \binom{[t]}{k-1}$,
\[
\deg_{\K}(S)\ge \left(\frac 1{k} - \r - d - \e \right)t - (k-1) \ge \left(\frac 1{k} - 2\r \right)t,
\]
because $d=\r/2$, $\e<\r/4$ and $k-1< \r t_0/4\le \r t/4$.
Let $m$ be the size of the clusters, then $(1-\e^2)\frac nt\le m\le \frac nt$. Applying Lemma \ref{almost} with the constants chosen above, we derive that
either there is a matching $M$ in $\K$ which covers all but at most $\a t$ vertices of $\K$ or there exists a set $B\subseteq V(\K)$, such that $|B| = \lfloor \frac{k-1}{k}t \rfloor $ and $e_{\K}(B)\le \r t^{k}$. In the latter case, let $B'\subseteq V(H)$ be the union of the clusters in $B$. By regularity,
\[
e_{H}(B')\le e_{\K}(B)\cdot m^k + \binom tk \cdot d \cdot m^k + \e^2 \cdot \binom tk \cdot m^k + t \binom m2 \binom n{k-2},
\]
where the right-hand side bounds the number of edges from regular $k$-tuples with high density, edges from regular $k$-tuples with low density, edges from irregular $k$-tuples and edges that lie in at most $k-1$ clusters. 
Since $m\le \frac nt$, $\e \ll \r$, $d=\r/2$, and $t^{-1}< t_0^{-1}< \r/(4k)$, we obtain that
\begin{align*}
e_{H}(B')\le \r t^k \cdot \left( \frac nt \right)^k + \binom tk \frac{\r}{2}  \left( \frac nt \right)^k  + \frac{\r}{16} \binom tk \left( \frac nt \right)^k + t \binom{n/t}2 \binom n{k-2} <\frac32 \r n^k.
\end{align*}
Note that $|B'|= \lfloor \frac{k-1}{k}t \rfloor m \le \frac{k-1}{k}t \cdot \frac nt= \frac{k-1}{k} n$, and consequently $|B'|\le \lfloor \frac{k-1}{k}n \rfloor$. On the other hand,
\begin{align*}
|B'| &= \left\lfloor \frac{k-1}{k}t \right\rfloor m\ge \left( \frac{k-1}{k}t-1 \right)(1-\e^2)\frac nt \ge \left( \frac{k-1}{k}t - \e^2 \frac{k-1}{k}t - 1 \right)\frac nt\\
&\ge \left( \frac{k-1}{k}t-\e^2 t \right)\frac nt =\frac{k-1}{k}n-\e^2 n.
\end{align*}
By adding at most $\e^2 n$ vertices from $V\setminus B'$ to $B'$, we get a set $B''\subseteq V(H)$ of size exactly $\lfloor \frac{k-1}{k}n \rfloor$, with $e(B'')\le e(B') + \e^2 n \cdot n^{k-1}<2\r n^k$. Hence $H$ is $2\r$-extremal.

In the former case, the union of the clusters covered by $M$ contains all but at most $\a t m+|V_0|\le \a n+ \e^2 n$ vertices of $H$. 
We apply the following procedure to each member $e\in M$. 
Note that the corresponding set of clusters $V_{i_1}, \dots, V_{i_k}$ for $e$ forms an $\e^2$-regular $k$-tuple.
By \cite{erdos}, we can greedily find vertex-disjoint copies of $K_k^k(h)$ until the regularity does not hold, i.e., the set of uncovered vertices in each $V_{i_j}$ has size at most $\e^2 m$.
Since $|M|\le \frac t{k}$, we thus obtain $K_k^k(h)$-tiling of $H$ covering all but at most
\[
k\e^2 m \cdot \frac{t}{k}+\a n+\e^2 n < 2\e^2 n + \a n  < 2\a n
\]
vertices of $H$, as $\e\ll \a$. This completes the proof.
\end{proof}

\section{The Extremal Case}

In this section we prove Theorem \ref{thm:E}. Take $0< {\e}\ll 1$ and let $n\in 6\mathbb{N}$ be sufficiently large. Let $\e_0=24\e$. Let $H=(V, E)$ be an $n$-vertex 3-graph with $\delta_2(H)\ge n/3$ which is ${\e}$-extremal, namely, there exists a set $B\subseteq V(H)$ of size $2n/3$ and
\begin{equation}\label{eqB}
e(B)\le \e n^3 = \frac{27}8 \e |B|^3 \le \e_0 \binom{|B|}3.
\end{equation}

Let $\e_1=8\sqrt{\e_0}$ and $A= V(H)\setminus B$. Assume that the partition $A$ and $B$ satisfies that $|B| = 2n/3$ and \eqref{eqB}. In addition, assume that $e(B)$ is the smallest among all the partitions satisfying these conditions. We now define
\begin{align*}
&A':=\left\{ v\in V\mid \deg (v,B)\ge (1-\e_1)\binom{|B|}{2} \right\}, \\
&B':=\left\{ v\in V\mid\deg (v,B)\le \e_1\binom{|B|}{2} \right\}, \\
& V_0=V\setminus(A'\cup B').
\end{align*}

The following simple claim appeared in \cite{HZ1}. We include its proof for completeness.

\begin{claim}\label{clm:eB}
$A\cap B'\neq \emptyset$ implies that $B\subseteq B'$, and $B\cap A'\neq \emptyset$ implies that $A\subseteq A'$.
\end{claim}

\begin{proof}
First, assume that $A\cap B'\neq \emptyset$. Then there is some $u\in A$ which satisfies that $\deg(u,B)\le \e_1\binom{|B|}2$. If there exists some $v\in B\setminus B'$, namely, $\deg(v, B)>\e_1\binom{|B|}2$, then we can switch $u$ and $v$ and form a new partition $A''\cup B''$ such that $|B''|=|B|$ and $e(B'')<e(B)$, which contradicts the minimality of $e(B)$.

Second, assume that $B\cap A'\neq \emptyset$. Then some $u\in B$ satisfies that $\deg(u,B)\ge (1-\e_1)\binom{|B|}2$. Similarly, by the minimality of $e(B)$, we get that for any vertex $v\in A$, $\deg(v, B)\ge (1-\e_1)\binom{|B|}2$, which implies that $A\subseteq A'$.
\end{proof}

\begin{claim}\label{clm:size}
$\{|A\setminus A'|, |B\setminus  B'|, |A'\setminus  A|, |B'\setminus  B|\}\le \frac{\e_1}{64}|B|$ and $|V_0|\le \frac{\e_1}{32}|B|$.
\end{claim}

\begin{proof}
First assume that $|B\setminus B'|> \frac{\e_1}{64}|B|$. By the definition of $B'$ and the assumption $\e_1=8\sqrt{\e_0}$, we get that
\[
e(B) > \frac 13 \e_1\binom {|B|}2 \cdot \frac{\e_1}{64}|B| > \frac{\e_1^2}{64}\binom {|B|}3  = \e_0 \binom {|B|}3,
\]
which contradicts \eqref{eqB}.

Second, assume that $|A\setminus A'|> \frac{\e_1}{64}|B|$. Then by the definition of $A'$, for any vertex $v\notin A'$, we have that $\overline \deg(v,B)> \e_1\binom{|B|}{2}$. So we get
\[
\overline e(ABB)> \frac{\e_1}{64}|B| \cdot \e_1 \binom{|B|}2=\e_0|B| \binom{|B|}2> 3\e_0 \binom{|B|}3.
\]
Together with \eqref{eqB}, this implies that
\begin{align*}
\sum_{b_1 b_2\in \binom B2}\overline\deg(b_1 b_2)&\ge 3\overline e(B)+\overline e(ABB) 
                              > 3(1 - \e_0) \binom{|B|}3 + 3\e_0 \binom{|B|}3 = \binom{|B|}2 (|B|-2).
\end{align*}
By the pigeonhole principle, there exists $b_1 b_2\in \binom B2$, such that $\overline\deg(b_1 b_2)  > |B|-2 = \frac{2n}3-2$,
contradicting $\delta_2(H)\ge n/3$.

Consequently,
\begin{align*}
&|A'\setminus A|=|A'\cap B|\le |B\setminus B'|\le \frac{\e_1}{64}|B|,   \\
&|B'\setminus B|=|A\cap B'|\le |A\setminus A'|\le \frac{\e_1}{64}|B|, \\
&|V_0|=|A\setminus A'|+|B\setminus B'|\le \frac{\e_1}{64}|B|+\frac{\e_1}{64}|B|=\frac{\e_1}{32}|B|.
\end{align*}
So the proof is complete.
\end{proof}

We first deal with a special (ideal) case of Theorem~\ref{thm:E}.

\begin{lemma}\label{lem3}
Let  $0<\rho \ll 1$ and let $n$ be sufficiently large. 
Suppose $H$ is a 3-graph on $n\in 6\mathbb{Z}$ vertices with a partition of $V(H)=X\cup Z$ such that $|Z|= 2|X|$. Furthermore, assume that
\begin{itemize}
\item for every vertex $v\in X$, $\overline \deg(v, Z)\le \rho\binom{|Z|}2$,
\item given every vertex $u\in Z$, we have $\overline{\deg}(u v, X)\le \rho |X|$, for all but at most $\rho |Z|$ vertices $v\in Z\setminus \{u\}$. 
\end{itemize}
Then $H$ contains a $C_6$-factor.
\end{lemma}

\medskip
To prove Lemma \ref{lem3}, we follow the approach in the proof of \cite[Lemma 3.4]{CzMo} given by Czygrinow and Molla, who applied a result of K\"uhn and Osthus \cite{KuOs06_pse}.
A bipartite graph $G=(A, B, E)$ with $|A|=|B|=n$ is called $(d, \e)$-\emph{regular} if for any two subsets $A'\subseteq A$, $B'\subseteq B$ with $|A'|, |B'|\ge \e n$,
\[
(1-\e)d\le \frac{e(A', B')}{|A'| |B'|}\le (1+\e)d,
\]
and $G$ is called \emph{$(d, \e)$-super-regular} if in addition $(1-\e)d n \le \deg(v) \le (1+\e) d n$ for every $v\in A\cup B$.

\begin{lemma}[\cite{KuOs06_pse}, Theorem 1.1]\label{lem:random}
For all positive constants $d, \nu_0, \eta\le 1$ there is a positive $\e=\e(d, \nu_0, \eta)$ and an integer $N_0$ such that the following holds for all $n\ge N_0$ and all $\nu\ge \nu_0$. Let $G=(A, B, E)$ be a $(d, \e)$-super-regular bipartite graph whose vertex classes both have size $n$ and let $F$ be a subgraph of $G$ with $|F|= \nu |E|$. Choose a perfect matching $M$ uniformly at random in $G$. Then with probability at least $1-e^{-\e n}$ we have
\[
(1-\eta) \nu n \le |M\cap E(F)| \le (1+\eta) \nu n.
\]
\end{lemma}

\medskip
\begin{proof}[Proof of Lemma \ref{lem3}.]
Let $\e = \e(1, 0.9, 0.01)$ be the constant returned by Lemma \ref{lem:random} and let $\rho \ll \e$.
Suppose that $n$ is sufficiently large and $H$ is a 3-graph satisfying the assumption of the lemma.
Let $G$ be the graph of all pairs $uv$ in $Z$ such that $\overline{\deg}(u v, X)\le \rho |X|$. By the assumption, for any vertex $v\in Z$, we know
\begin{equation}\label{eq:dGY1}
\overline{\deg}_G(v)\le \rho |Z|.
\end{equation}

Let $m=|X|/2=|Z|/4$ and note that $m\in \mathbb{Z}$.
Arbitrarily 
partition $Z$ into four sets $Z_1, Z_2, Z_3, Z_4$, each of order $m$. 
Let $M=\{x_1 x_1', \dots, x_m x_m'\}$ be an arbitrary perfect matching of $X$.
By \eqref{eq:dGY1} and $|Z|= 4m$, we have
$\delta(G[Z_i, Z_{i+1}])\ge (1-4\rho)m$ for $i\in [3]$. It is easy to see that for $i\in [3]$, $G[Z_i, Z_{i+1}]$ is $(1, \e)$-super-regular as $\rho \ll \e$.
For any $x\in X$ and $i\in [3]$, let $F_x^i:= E(G[Z_i, Z_{i+1}]) \cap N_H(x)$.
Since $\overline \deg(x, Z)\le \rho\binom{|Z|}2\le 8\rho m^2$, we have $|F_x^1|, |F_x^2|, |F_x^3|\ge (1-4\rho) m^2 - 8\rho m^2 \ge 0.9 m^2$, as $\rho \ll 1$.
For $i\in [3]$, let $M_i$ be a perfect matching chosen uniformly at random from $G[Z_i, Z_{i+1}]$.
By applying Lemma \ref{lem:random} with $\nu_0 =0.9$ and $\eta=0.01$, for any $x\in X$, with probability at least $1-e^{-\e m}$, we have
\begin{equation}\label{eq:m12}
|M_1\cap E(F_x^1)|, |M_2\cap E(F_x^2)|, |M_3\cap E(F_x^3)| \ge (1-\eta) \nu_0 m \ge 0.89 m.
\end{equation}
Thus for all $i\in [3]$, there exists a matching $M_i$ in $G[Z_i, Z_{i+1}]$ such that \eqref{eq:m12} holds for all $x\in X$.
Label $Z_i=\{z_1^i,\dots, z_m^i\}$ for $i\in [4]$ such that $M_i=\{z_1^i z_1^{i+1}, \dots, z_{m}^i z_{m}^{i+1}\}$.
Let $\Gamma$ be a bipartite graph on $(M, [m])$ such that $\{x_j x_j', i\}\in E(\Gamma)$ if and only if 
\[
x_j z_i^1 z_i^2, x_j' z_i^2 z_i^3, x_j z^3_i z_i^4\in E(H)
\]
for $x_j x_j'\in M$ and $i\in [m]$. For every $i\in [m]$, since $z_i^1 z_i^2, z_i^2 z_i^3, z^3_i z_i^4\in E(G)$, we have $\deg_\Gamma(i) \ge m - 3\rho |X| = (1-6\rho)m$ by the definition of $G$. On the other hand, by \eqref{eq:m12}, we have $\deg_\Gamma(x_j x_j') \ge m - 3(1- 0.89) m = 0.67 m$ for any $x_j x_j'\in M$. By a simple corollary of Hall's Theorem, $\Gamma$ contains a perfect matching, which gives a $C_6$-factor in $H$.
\end{proof}

Now we are ready to prove Theorem \ref{thm:E}.

\begin{proof}[Proof of Theorem \ref{thm:E}.]
We will build four vertex-disjoint $C_6$-tilings $\Q_1, \Q_2, \R, \sss$ whose union is a perfect $C_6$-tiling of $H$. 
The purpose of the $C_6$-tilings $\Q_1, \Q_2, \R$ is covering the vertices of $V_0$ and adjusting the sizes of $A'$ and $B'$ such that we can apply Lemma \ref{lem3} after $\Q_1, \Q_2, \R$ are removed.
Note that by $|B\setminus  B'|\le \frac{\e_1}{64}|B|$,
\begin{equation}\label{eq:V_0}
\deg(w, B')\ge \deg\left(w,B \right)-|B\setminus B'| |B| \ge \frac{\e_1}2\binom{|B'|}{2} \text{ for any vertex }w\in V_0,
\end{equation}
and by $|B'\setminus  B|\le \frac{\e_1}{64}|B|$, we have
\begin{equation*}
\deg(v, B')\le \deg\left(v,B \right)+|B'\setminus B| |B'| \le 2\e_1\binom{|B'|}{2} \text{ for any vertex }v\in B'.
\end{equation*}
Moreover, the latter inequality implies that for all but at most $\sqrt{2\e_1} |B'|$ vertices $u\in B'$, we have $\deg(u v, B') \le \sqrt{2\e_1} |B'|$.
By $\delta_2(H)\ge n/3$ and the bounds in Claim \ref{clm:size}, this implies that
\begin{equation*}
\deg(u v, A') \ge n/3 - \deg(u v, B') - |V_0| \ge n/3 - 2\sqrt{\e_1} |B|.
\end{equation*}
By Claim \ref{clm:size}, $|A'|\le n/3 + \frac{\e_1}{64}|B|$ and thus
\begin{equation*}
\overline\deg(u v, A') \le 2\sqrt{\e_1} |B| + \frac{\e_1}{64}|B| \le 3\sqrt{\e_1} |B| = 2\sqrt{\e_1} n.
\end{equation*}
So we have the following

\medskip
\noindent (\dag) given every vertex $v\in Z$, we have $\overline{\deg}(u v, A')\le 2\sqrt{\e_1} n$, for all but at most $\sqrt{2\e_1} |B'|$ vertices $u\in B'\setminus \{v\}$. 
\medskip

\medskip
\noindent {\bf The $C_6$-tilings $\Q_1, \Q_2$.} 
Assume that $|V_0|=q_1$ and $|B'| = 2n/3 + q$. Thus, by Claim \ref{clm:size}, $q_1, q\in \mathbb{Z}$ and $0\le q_1\le \frac{\e_1}{32}|B|$, $-\frac{\e_1}{64}|B|\le q\le \frac{\e_1}{64}|B|$. 
We claim that there is a $C_6$-tiling $\Q_1$ consisting of $q_1$ copies of $C_6$ such that each copy contains one vertex in $A'$, one vertex in $V_0$ and four vertices in $B'$ and a $C_6$-tiling $\Q_2$ consisting of $q_2 = \max\{q, 0\}$ copies of $C_6$ such that each copy contains one vertex in $A'$ and five vertices in $B'$. 

To see this, first, note that $\delta_2(H[B'])\ge q_2$. Thus, by a result of \cite[Fact 2.1]{RRS09}, we know that $B'$ contains a matching $M=\{e_1,\dots, e_{q_2}\}.$\footnote{We remark that this is the only place where we need the exact codegree condition $n/3$.} 
Now consider $V_0\cup (B'\setminus V(M))$. We claim that we can greedily find a matching $M'=\{e_{q_2+1},\dots, e_{q_1+q_2}\}$ such that each edge contains exactly one vertex in $V_0$.
Indeed, by \eqref{eq:V_0}, each vertex $w\in V_0$ has at least $\frac{\e_1}2\binom{|B'|}{2}$ neighbors in $B'$. Note that the number of vertices in the existing matching is at most $3(q_1+q_2)\le \frac{9\e_1}{64}|B|$, and thus the number of pairs that are unavailable for $w$ is at most $\frac{9\e_1}{64}|B| |B'| < \frac{\e_1}2\binom{|B'|}{2}$. So we can pick an edge that contains $w$ and two vertices in $B'$ which is disjoint from other edges in the matching.

For each $1\le i\le q_1 + q_2$, let $e_i = \{u_i, v_i, w_i\}$. In particular, assume $V_0=\{w_{q_2+1}, \dots, w_{q_1 + q_2}\}$. 
By (\dag), fix $u_i, v_i\in B'$, we can pick vertices $x_i, y_i\in B'$ such that $\overline\deg(u_i x_i, A') \le 2\sqrt{\e_1} n$ and $\overline\deg(v_i y_i, A') \le 2\sqrt{\e_1} n$. So we can pick a vertex $z_i\in N(u_i x_i, A')\cap N(v_i y_i, A')$. 
Note that $\{u_i, v_i, w_i, x_i, y_i, z_i\}$ spans a desired copy of $C_6$.
Also, note that we have $|B'| - \sqrt{2\e_1} |B'|$ choices for each $x_i$ and $y_i$, respectively, and $|A'| - 4\sqrt{\e_1}n$ choices for $z_i\in A'$.
So we can select these vertices without repetition, which gives the desired $C_6$-tilings $\Q_1$ and $\Q_2$.

Let $A_1$ and $B_1$ be the sets of vertices in $A'$ and $B'$ not covered by $\Q_1\cup \Q_2$, respectively. Note that $q_1 + q_2\le |A\setminus A'|\le \frac{\e_1}{64}|B|$ by Claim~\ref{clm:size}, thus
\begin{equation}\label{eq:B_1}
|B_1|\ge |B'|-5(q_1+q_2)\ge |B'|-\frac{5\e_1}{64}|B|\ge |B| - \frac{3\e_1}{32}|B|.
\end{equation}
Note that by $|A'|+|V_0|+|B'|=n$, we have $|A'| = \frac n3 - q_1 - q$.
If $q\ge 0$, then by the definition of $\Q_1, \Q_2$, we have $|A_1|=|A'| - q_1 - q = \frac n3 - 2q_1 - 2q$ and $|B_1| = |B'| - 4q_1 - 5q = \frac{2n}3 - 4q_1 - 4q$, i.e., $|B_1| = 2|A_1|$.
Otherwise $q < 0$ and thus, $|A_1|=|A'| - q_1 = \frac n3 - 2q_1 - q$ and $|B_1| = |B'| - 4q_1 = \frac{2n}3 - 4q_1 + q$. So we have $2|A_1| - |B_1| = -3q >0$.
Define $s=\frac13(2|A_1| - |B_1|)$. Then 
$s=0$ if $q\ge 0$ and $s=-q \le \frac{\e_1}{64}|B|$ if $q<0$.

\medskip
\noindent {\bf The $C_6$-tiling $\R$.} Next we build our $C_6$-tiling $\R$ of size $s \le \frac{\e_1}{64}|B|$ such that every element of $\R$ contains three vertices in $A_1$ and three vertices in $B_1$. 
We will construct one desired copy of $C_6$ such that for each of its vertex $v$, there are more than $3s$ vertices in $A_1$ or $B_1$ can be selected as $v$, thus proving the claim.
We start with any vertex $u$ in $B_1$. 
By (\dag), we can pick $v\in B_1$ and then pick $w\in B_1$ such that $\overline\deg(u w, A_1) \le 2\sqrt{\e_1} n$ and $\overline\deg(v w, A_1) \le 2\sqrt{\e_1} n$. 
Note that the numbers of choices for $v$ and $w$ are at least $|B_1| - \sqrt{2\e_1} |B'| >3s$ and at least $|B_1| - 2\sqrt{2\e_1} |B'|>3s$, respectively.
At last we pick $x\in N(uv, A_1)$, $y\in N(uw, A_1)$ and $z\in N(vw, A_1)$, and for each of them, at least $|A_1| - 2\sqrt{\e_1} n > 3s$ vertices can be selected. This completes the proof.

\medskip
Let $A_2$ be the set of vertices of $A$ not covered by $\Q_1, \Q_2, \R$ and define $B_2$ similarly. 
Then $|A_2|=|A_1|-3s$ and $|B_2|=|B_1|-3s$. If $q\ge 0$, then $s=0$ and $|B_2|=2|A_2|$. Otherwise $s=-q$ and so $2|A_2| - |B_2| = 2|A_1| - |B_1| - 3s = 0$. Furthermore, by $s\le \frac{\e_1}{64}|B|$ and \eqref{eq:B_1}, we have 
\[
|B_2|= |B_1|-3s\ge |B| - \frac{3\e_1}{32}|B| - \frac{3\e_1}{64}|B|> (1-\e_1)|B|.
\]
Hence, for every vertex $v\in A_2$,
\[
\overline{\deg}(v, B_2)\le \overline{\deg}(v, B')\le \e_1 \binom {|B|}2 + |B'\setminus B| |B'|\le \e_1 \binom {\frac{1}{1-\e_1}|B_2|}2 + \frac{\e_1}{2}|B_2|^2< 3\e_1 \binom{|B_2|}2.
\]
Moreover, by the definition of $A_2$ and $s$, we have
\[
|A_2|=|A_1|-3s \ge \frac n3 - 2q_1 - 2|q| - 3s \ge (1/3-\e_1)n. 
\]
So we get $n\le 4|A_2|$.
By (\dag), given $v\in B_2$, for all but at most $\sqrt{2\e_1} |B'|\le 2\sqrt{\e_1} |B_2|$ vertices $u\in B_2$, we have
\begin{align*}
\overline{\deg}(u v, A_2)\le 2\sqrt{\e_1} n \le 8\sqrt{\e_1} |A_2|.
\end{align*}

\medskip
\noindent {\bf The $C_6$-tiling $\sss$.} 
At last, we apply Lemma \ref{lem3} with $X=A_2$, $Z=B_2$ and $\rho=8\sqrt{\e_1}$ and get a $C_6$-factor $\sss$ on $A_2\cup B_2$.
This concludes the proof of Theorem \ref{thm:E}.
\end{proof}

\section{Concluding Remarks}

In this paper we have studied $C_6$-factors in 3-graphs.
Note that we can state our main result in the following way: Given $n=6t$ be sufficiently large, then any $n$-vertex 3-graph $H$ with $\delta_2(H)\ge 2t$ contains $t$ vertex-disjoint copies of $C_6$.
This suggests the following conjecture.

\begin{conjecture}\label{conj:1}
Given $n\ge 6t$ be sufficiently large, then any $n$-vertex 3-graph $H$ with $\delta_2(H)\ge 2t$ contains $t$ vertex-disjoint copies of $C_6$.
\end{conjecture}

Note that this conjecture, if true, trivially implies the following conjecture.

\begin{conjecture}\label{conj:2}
Given $n\ge 6t$ be sufficiently large, then any $n$-vertex 3-graph $H$ with $\delta_2(H)\ge 2t$ contains $t$ vertex-disjoint loose cycles.
\end{conjecture}

Conjecture~\ref{conj:2} can be seen as an analogue of Corr\'adi-Hajnal Theorem for loose cycles in 3-graphs.
It is not hard to show both conjectures for $t=1$.

Note that the result in \cite{HZZ_tiling} implies that $t_1(n, C_6^3) = (5/9+o(1))\binom n2$.
Indeed, it is shown that $t_1(n, K_3^3(2)) = (5/9+o(1))\binom n2$ and the upper bound holds because $C_6^3$ is a subhypergraph of $K_3^3(2)$.
The lower bound follows from the construction that shows the sharpness of Theorem~\ref{main} in Section 1. It is interesting to know the exact value of $t_1(n, C_6^3)$.

\section*{Acknowledgment}

We thank Yi Zhao for invaluable comments on the manuscript.

\bibliographystyle{plain}
\bibliography{Apr2015}

\begin{thebibliography}{10}

\bibitem{AFHRRS}
N.~Alon, P.~Frankl, H.~Huang, V.~R{\"o}dl, A.~Ruci{\'n}ski, and B.~Sudakov.
\newblock Large matchings in uniform hypergraphs and the conjecture of {E}rd{\H
  o}s and {S}amuels.
\newblock {\em J. Combin. Theory Ser. A}, 119(6):1200--1215, 2012.

\bibitem{AY96}
N.~Alon and R.~Yuster.
\newblock {$H$}-factors in dense graphs.
\newblock {\em J. Combin. Theory Ser. B}, 66(2):269--282, 1996.

\bibitem{CoHa}
K.~Corr\'adi and A.~Hajnal.
\newblock On the maximal number of independent circuits in a graph.
\newblock {\em Acta Math. Acad. Sci. Hungar.}, 14:423--439, 1963.

\bibitem{Czy14}
A.~Czygrinow.
\newblock Minimum degree condition for {$C_4$}-tiling in 3-uniform hypergraphs.
\newblock {\em submitted}.

\bibitem{CDN}
A.~Czygrinow, L.~DeBiasio, and B.~Nagle.
\newblock Tiling 3-uniform hypergraphs with ${K}_4^3-2e$.
\newblock {\em Journal of Graph Theory}, 75(2):124--136, 2014.

\bibitem{CzKa}
A.~Czygrinow and V.~Kamat.
\newblock Tight co-degree condition for perfect matchings in 4-graphs.
\newblock {\em Electron. J. Combin.}, 19(2):Paper 20, 16, 2012.

\bibitem{CzMo}
A.~Czygrinow and T.~Molla.
\newblock Tight codegree condition for the existence of loose {H}amilton cycles
  in 3-graphs.
\newblock {\em SIAM J. Discrete Math.}, 28(1):67--76, 2014.

\bibitem{erdos}
P.~Erd\H{o}s.
\newblock On extremal problems of graphs and generalized graphs.
\newblock {\em Israel Journal of Mathematics}, 2(3):183--190, 1964.

\bibitem{HaSz}
A.~Hajnal and E.~Szemer{\'e}di.
\newblock Proof of a conjecture of {P}. {E}rd{\H o}s.
\newblock In {\em Combinatorial theory and its applications, {II} ({P}roc.
  {C}olloq., {B}alatonf\"ured, 1969)}, pages 601--623. North-Holland,
  Amsterdam, 1970.

\bibitem{HS}
H.~H\`an and M.~Schacht.
\newblock Dirac-type results for loose {Hamilton} cycles in uniform
  hypergraphs.
\newblock {\em Journal of Combinatorial Theory. Series B}, 100:332--346, 2010.

\bibitem{Han14_poly}
J.~Han.
\newblock Decision problem for perfect matchings in dense uniform hypergraphs.
\newblock {\em submitted}.

\bibitem{Han14_mat}
J.~Han.
\newblock Near perfect matchings in $k$-uniform hypergraphs.
\newblock {\em Combinatorics, Probability and Computing}, 24:723--732, 9 2015.

\bibitem{HLTZ_K4}
J.~Han, A.~Lo, A.~Treglown, and Y.~Zhao.
\newblock Exact minimum codegree threshold for {$K_4^-$}-factors.
\newblock {\em preprint}.

\bibitem{HZZ_tiling}
J.~Han, C.~Zang, and Y.~Zhao.
\newblock Minimum vertex degree thresholds for tiling complete $3$-partite
  $3$-graphs.
\newblock {\em submitted}.

\bibitem{HZ_cycpac}
J.~Han and Y.~Zhao.
\newblock Minimum codegree thresholds for tiling loose cycles in $k$-uniform
  hypergraphs.
\newblock in preparation.

\bibitem{HZ1}
J.~Han and Y.~Zhao.
\newblock Minimum degree thresholds for loose {Hamilton} cycle in 3-graphs.
\newblock {\em Journal of Combinatorial Theory, Series B}, 114:70 -- 96, 2015.

\bibitem{HZ3}
J.~Han and Y.~Zhao.
\newblock Minimum vertex degree threshold for {$C_4^3$}-tiling.
\newblock {\em Journal of Graph Theory}, 79(4):300--317, 2015.

\bibitem{HeKi}
P.~Hell and D.~G. Kirkpatrick.
\newblock On the complexity of general graph factor problems.
\newblock {\em SIAM J. Comput.}, 12(3):601--609, 1983.

\bibitem{Keevash_blowup}
P.~Keevash.
\newblock A hypergraph blow-up lemma.
\newblock {\em Random Structures Algorithms}, 39(3):275--376, 2011.

\bibitem{KM1}
P.~Keevash and R.~Mycroft.
\newblock A geometric theory for hypergraph matching.
\newblock {\em Memoirs of the American Mathematical Society}, 233(Monograph
  1908), 2014.

\bibitem{Khan2}
I.~Khan.
\newblock Perfect matchings in 4-uniform hypergraphs.
\newblock {\em arXiv:1101.5675}.

\bibitem{Khan1}
I.~Khan.
\newblock Perfect matchings in 3-uniform hypergraphs with large vertex degree.
\newblock {\em SIAM J. Discrete Math.}, 27(2):1021--1039, 2013.

\bibitem{KSS-AY}
J.~Koml{\'o}s, G.~S{\'a}rk{\"o}zy, and E.~Szemer{\'e}di.
\newblock Proof of the {A}lon-{Y}uster conjecture.
\newblock {\em Discrete Math.}, 235(1-3):255--269, 2001.
\newblock Combinatorics (Prague, 1998).

\bibitem{KO}
D.~K\"uhn and D.~Osthus.
\newblock Loose {Hamilton} cycles in 3-uniform hypergraphs of high minimum
  degree.
\newblock {\em Journal of Combinatorial Theory. Series B}, 96(6):767--821,
  2006.

\bibitem{KuOs06_pse}
D.~K{\"u}hn and D.~Osthus.
\newblock Multicolored {H}amilton cycles and perfect matchings in pseudorandom
  graphs.
\newblock {\em SIAM J. Discrete Math.}, 20(2):273--286 (electronic), 2006.

\bibitem{KuOs-survey}
D.~K{\"u}hn and D.~Osthus.
\newblock Embedding large subgraphs into dense graphs.
\newblock In {\em Surveys in combinatorics 2009}, volume 365 of {\em London
  Math. Soc. Lecture Note Ser.}, pages 137--167. Cambridge Univ. Press,
  Cambridge, 2009.

\bibitem{KuOs09}
D.~K{\"u}hn and D.~Osthus.
\newblock The minimum degree threshold for perfect graph packings.
\newblock {\em Combinatorica}, 29(1):65--107, 2009.

\bibitem{KOT}
D.~K{\"u}hn, D.~Osthus, and A.~Treglown.
\newblock Matchings in 3-uniform hypergraphs.
\newblock {\em J. Combin. Theory Ser. B}, 103(2):291--305, 2013.

\bibitem{LM2}
A.~Lo and K.~Markstr{\"o}m.
\newblock Minimum codegree threshold for {$(K_4^3-e)$}-factors.
\newblock {\em J. Combin. Theory Ser. A}, 120(3):708--721, 2013.

\bibitem{LM1}
A.~Lo and K.~Markstr\"om.
\newblock F-factors in hypergraphs via absorption.
\newblock {\em Graphs and Combinatorics}, 31(3):679--712, 2015.

\bibitem{My14}
R.~Mycroft.
\newblock Packing k-partite k-uniform hypergraphs.
\newblock {\em Journal of Combinatorial Theory, Series A, accepted}.

\bibitem{RR}
V.~R\"odl and A.~Ruci\'nski.
\newblock Dirac-type questions for hypergraphs — a survey (or more problems
  for endre to solve).
\newblock {\em An Irregular Mind}, Bolyai Soc. Math. Studies 21:561--590, 2010.

\bibitem{RRS06}
V.~R\"odl, A.~Ruci\'nski, and E.~Szemer\'edi.
\newblock A {D}irac-type theorem for 3-uniform hypergraphs.
\newblock {\em Combinatorics, Probability and Computing}, 15(1-2):229--251,
  2006.

\bibitem{RRS09}
V.~R{\"o}dl, A.~Ruci{\'n}ski, and E.~Szemer{\'e}di.
\newblock Perfect matchings in large uniform hypergraphs with large minimum
  collective degree.
\newblock {\em J. Combin. Theory Ser. A}, 116(3):613--636, 2009.

\bibitem{Sze}
E.~Szemer{\'e}di.
\newblock Regular partitions of graphs.
\newblock In {\em Probl\`emes combinatoires et th\'eorie des graphes ({C}olloq.
  {I}nternat. {CNRS}, {U}niv. {O}rsay, {O}rsay, 1976)}, volume 260 of {\em
  Colloq. Internat. CNRS}, pages 399--401. CNRS, Paris, 1978.

\bibitem{TrZh15}
A.~Treglown and Y.~Zhao.
\newblock {A note on perfect matchings in uniform hypergraphs}.
\newblock {\em preprint}.

\bibitem{TrZh13}
A.~Treglown and Y.~Zhao.
\newblock Exact minimum degree thresholds for perfect matchings in uniform
  hypergraphs {II}.
\newblock {\em J. Combin. Theory Ser. A}, 120(7):1463--1482, 2013.

\bibitem{Tu47}
W.~T. Tutte.
\newblock The factorization of linear graphs.
\newblock {\em J. London Math. Soc.}, 22:107--111, 1947.

\bibitem{zsurvey}
Y.~Zhao.
\newblock Recent advances on dirac-type problems for hypergraphs.
\newblock {\em preprint}.

\end{thebibliography}

\end{document}